\documentclass[12pt]{amsart}
\usepackage{color}
\usepackage{amssymb}
\usepackage{comment}
\usepackage{pdfpages}
\usepackage{graphicx}
\usepackage{dcolumn}

\RequirePackage[numbers]{natbib}
\RequirePackage[colorlinks=true, pdfstartview=FitV, linkcolor=blue,
  citecolor=blue, urlcolor=blue]{hyperref}
\usepackage{paralist}

\usepackage{tikz} 
\usepackage{dcolumn}
\usetikzlibrary{arrows,positioning}
\tikzset{
    >=stealth',
    pil/.style={
           ->,
           thick,
           shorten <=2pt,
           shorten >=2pt,}
}

\usepackage{graphicx}
\graphicspath{{./figures/}}
\usepackage{amsfonts}
\usepackage{amsmath}
\usepackage{amsthm}
\usepackage{amssymb}
\usepackage{amsbsy}
\usepackage{epsfig}
\usepackage{fullpage}
\usepackage{natbib, mathrsfs} 
\usepackage{verbatim}
\usepackage[latin1]{inputenc}
\usepackage{mhequ}
\usepackage{algorithm}
\usepackage{algorithmic}

\numberwithin{equation}{section}

\def \be{\begin{equs}}
\def \ee{\end{equs}}
\def \P{\mathbb{P}}
\def \E{\mathbb{E}}

\def \so {\mathrm{SO}(n)}
\def \sn{\mathrm{S}^{n-1}}
\def \TV{\mathrm{TV}}
\def \hs{\mathrm{HS}}
\def  \la{\langle}
\def  \ra{\rangle}
\def \unif{\mathrm{unif}}

\def \tmix{\tau_{\mathrm{mix}}}

\def \G{\mathcal{G}}
\def \Op{\mathrm{Op}}
\def \exp{\mathrm{exp}}

\def \t{\theta}
\def \tt{\tilde{\theta}}
\def \tR{\tilde{\R}}
\def \Tr{\mathrm{Tr}}
\def \I{\mathrm{Id}} 
\def \R{\mathrm{R}}
\def \L{\mathrm{L}}
\def \lag{\mathfrak{so}(n)}
\def \d{\mathrm{d}}

\newtheorem{theorem}{Theorem}[section]
\newtheorem{lemma}[theorem]{Lemma}
\newtheorem{remarks}[theorem]{Remarks}

\newtheorem{prop}[theorem]{Proposition}

\newtheorem{assumption}[theorem]{Assumption}

\theoremstyle{plain}
\newtheorem{thm}{Theorem}
\newtheorem*{thm-non}{Theorem}

\theoremstyle{definition}
\newtheorem{defn}[theorem]{Definition}
\newtheorem{remark}[theorem]{Remark}

\begin{document}

\title[Kac's walk on $\so$]{On the Mixing Time of Kac's Walk and Other High-Dimensional Gibbs Samplers with Constraints}


\author{Natesh S. Pillai$^{\ddag}$}
\thanks{$^{\ddag}$pillai@fas.harvard.edu, 
   Department of Statistics,
    Harvard University, 1 Oxford Street, Cambridge
    MA 02138, USA}

\author{Aaron Smith$^{\sharp}$}
\thanks{$^{\sharp}$smith.aaron.matthew@gmail.com, 
   Department of Mathematics and Statistics,
University of Ottawa, 585 King Edward Avenue, Ottawa
ON K1N 7N5, Canada}

\maketitle






\begin{abstract}
Determining the total variation mixing time of Kac's random walk on the special orthogonal group $\so$
has been a long-standing open problem. In this paper, we construct a novel non-Markovian coupling for bounding this mixing time. The analysis of our coupling entails controlling the smallest singular value of a certain random matrix with highly dependent entries. The dependence of the entries in our matrix makes it not-amenable to 
existing techniques in random matrix theory. To circumvent this difficulty, we extend some recent bounds on the smallest singular values of matrices with independent entries to our setting. These bounds imply that the mixing time of Kac's walk on the group $\so$ is between $C_{1} n^{2}$ and $C_{2} n^{4} \log(n)$ for some explicit constants $0 < C_{1}, C_{2} < \infty$, substantially improving on the bound of $O(n^{5} \log(n)^{2})$ in the preprint \cite{jiangSOpre} of Jiang. Our methods may also be applied to other high dimensional Gibbs samplers with constraints and thus are of independent interest. In addition to giving analytical bounds on the mixing time, our approach allows us to compute rigorous estimates of the mixing time by simulating the eigenvalues of a random matrix.\end{abstract}

\section{Introduction} \label{sec:intro}

Mark Kac introduced a random walk on the sphere in his 1954 paper \cite{Kac1954} as a model for a Boltzmann gas. In this paper, we study Kac's walk on the special orthogonal group $\so$, which was first introduced in a statistical context \cite{hastings1970monte} and has been studied as a generalization of Kac's walk on the sphere since \cite{diaconis2000} (see also, \textit{e.g.}, \cite{Oliv07,pak2007convergence,carlen2003,janvresse2003bounds,hough2012asymptotic}).

Kac's walk on $\so$ is a discrete-time Markov chain $\{X_{t}\}_{t \geq 0}$ that evolves as follows. Fix an ordering of the $N \equiv \frac{n(n-1)}{2}$ planes generated by two coordinates in $\mathbb{R}^{n}$ and choose $X_{0} \in \so$. For $t \in \mathbb{N}$, choose  $1 \leq i_{t} \leq N$ and $\theta \in [0,2 \pi]$ uniformly at random and set 
\be \label{KacRep}
X_{t+1} =  \R(i_{t},\theta_{t}) X_{t},
\ee 
where $\R(i,\theta)$ denotes a rotation by the angle $\theta$ in the $i$'th 
coordinate plane. If the $i$'th coordinate plane is associated with the coordinate axes $1 \leq k < \ell \leq n$, $\R(i,\theta)$ is an $n \times n$ matrix with entries
\be[eqn:R] 
\R(i,\theta)_{jj} &= \cos(\theta), \qquad \qquad j \in \{k, \ell\} \\
\R(i,\theta)_{k \ell} &=  \sin(\theta), \quad
\R(i,\theta)_{\ell k} =  -\sin(\theta) \\
\R(i,\theta)_{jj} &= 1, \qquad \qquad j \notin \{k,\ell\} \\
\R(i,\theta)_{j j'} &= 0, \qquad \qquad j' \notin \{j,k,\ell\}.
\ee 

If we write $X_{t} = [v_{t}^{(1)} \, v_{t}^{(2)} \, \ldots \, v_{t}^{(n)}]$, the law of $\{ v_{t}^{(1)} \}_{t \geq 0}$ is known as \textit{Kac's walk on the sphere $\sn$}. Physically, Kac motivated this random walk by considering $n$ particles in a one-dimensional box. He assumed that these particles were uniformly distributed in space, and the vector $v_{t}^{(1)}$ models the change in their velocities over time as collisions occur; the condition that $v_{t}^{(1)}$ be constrained to the sphere corresponds to the principle of conservation of energy. Understanding the mixing properties of this process is central to Kac's program in kinetic theory (see \cite{mischler2013kac} for a useful description of this program). 
Kac's walks on the sphere and on $\so$ have attracted great attention and estimating 
their mixing times has been a long standing open problem (see Sections \ref{sec:mot} and \ref{sec:litrev}). Recently, in \cite{pillai2015kac}, the authors of this paper
obtained a matching upper bound and lower bound for the mixing time of Kac's walk in $\mathrm{S}^{n-1}$, thus settling this problem upto a constant factor.

To state our main result, we recall some standard definitions. For measures $\nu_1, \nu_2$ on a measure space $(\Omega, \mathcal{F})$, the \textit{total variation distance} between $\nu_1, \nu_2$ is given by
\be 
\| \nu_1 - \nu_2 \|_{\TV} = \sup_{A \in \mathcal{F}} (\nu_1(A) - \nu_2(A)).
\ee 
We denote the distribution of a random variable $X$ by $\mathcal{L}(X)$ and write $X \sim \nu$ as a shorthand for $\mathcal{L}(X) = \nu$. For a Markov chain $\{X_{t} \}_{t \geq 0}$ with unique associated stationary distribution $\nu$ on state space $\Omega$, we define the associated \textit{mixing profile} by 
\be 
\tau(\epsilon) = \inf \{t \, : \, \sup_{X_{0} = x \in \Omega} \| \mathcal{L}(X_{t}) - \nu \|_{\TV} < \epsilon \} 
\ee 
and the \textit{mixing time} by $\tmix = \tau(0.25)$. 

Let $\mu$ denote the normalized Haar measure on $\so$.
Our main result is the following bound on the mixing time of Kac's walk on $\so$:

\begin{theorem} \label{MainKacThm}
Let $\{X_{t}\}_{t \geq 0}$ be a copy of Kac's walk on $\so$. Then for  all sequences $T = T(n) > 10^{7} \, n^{4} \log(n)$,
\be \label{IneqMainThmUpper}
\limsup_{n \rightarrow \infty} \sup_{X_{0} = x \in \so} \| \mathcal{L}(X_{T}) - \mu \|_{\TV} = 0,
\ee 
and for all sequences $T = T(n) < N$,
\be \label{IneqMainThmLower}
\liminf_{n \rightarrow \infty} \sup_{X_{0} = x \in \so} \| \mathcal{L}(X_{T}) - \mu \|_{\TV} = 1.
\ee 
\end{theorem}
We have not tried to optimize the constant $10^{7}$ appearing in Theorem \ref{MainKacThm}.

\subsection{Motivations Outside of Physics} \label{sec:mot}
Kac's random walk has been studied in a wide range of fields including computer science, statistics and numerical analysis.
To our knowledge, the Markov chain that we call Kac's walk on $\so$ was initially proposed in \cite{hastings1970monte} as a Gibbs sampler targetting the Haar measure on $\so$. The problem of sampling from Haar measure on $\so$ was motivated by \cite{james1955generating}, but the walk itself has been suggested as a computationally efficient method for finding projections onto random small-dimensional subspaces \cite{ailon2006approximate}. Bounds on the mixing time of Kac's walk are required to check that this approach is, in fact, computationally efficient.

Our analysis of Kac's walk is also interesting as a worked example that belongs to several active areas of research. The Markov chains we study are a sequence of high-dimensional Gibbs samplers (see \cite{casella1992explaining}). Despite three decades of extensive work in this area, there are few effective bounds on the mixing times of Gibbs samplers in high dimensions (see \cite{jones2001honest, diaconis2008gibbs} for an introduction to the large literature on this problem). Of the existing effective bounds on the mixing times of high dimensional Gibbs samplers on continuous state spaces, almost all target distributions with support equal to a union of quadrants of $\mathbb{R}^{n}$ (\textit{e.g.,} \cite{jones2001honest}) or involve explicitly computing spectral information for the transition kernel (\textit{e.g.,} the analyses \cite{rosenthal1994random,porod1996cut,hough2012asymptotic} of other random walks on $\so$). Our analysis gives one of relatively few results for Gibbs samplers on  a complicated sample space for which spectral information cannot easily be used. Some closely related papers are  \cite{lovasz1999hit,lovasz2003hit} on Gibbs samplers on convex sets, as well as the papers \cite{smith2014gibbs,jiang2012} directly motivated by the study of Kac's walk on $\so$. There is also a large number of papers studying the mixing time of Markov chains on groups \cite{diaconis1988group,rosenthal1994random, furman2002random,saloff2004random,hough2012asymptotic} using Fourier analysis and representation theory. Unlike these papers, we do not use Fourier analytic tools and thus our methods are generalizable to other Gibbs samplers in $\mathbb{R}^n$.

\subsection{Previous Work} \label{sec:litrev}

The central question in the study of Kac's walk is to determine the speed at which it converges to equilibrium. This question is somewhat vague, as it does not specify the metric under which convergence is to be measured. Early work focused on proving that the spectral gap of the chain was large. In \cite{diaconis2000}, the authors showed that the spectral gap of the walk on $\so$ was at least order of $n^{-3}$. Janvresse first showed in \cite{jan2001} that the spectral gap of the walk on the sphere was exactly on the order of $n^{-1}$. Janvresse also showed in \cite{janvresse2003bounds} that the spectral gap of the walk on $\so$ was on the order of $n^{-2}$. The exact spectral gap for both walks was found in \cite{carlen2003}, and the full spectrum was computed in \cite{maslen2003}. Some of this work was generalized in \cite{caputo2008spectral}. Although these bounds imply a convergence \textit{rate} for Kac's walk in $L^{2}$, and a bound on the \textit{distance} to stationarity in $L^{2}$ implies a bound on the total variance distance to stationarity, these bounds \textit{do not} imply any bound at all on the total variation mixing time of Kac's walk. This is because, when $\mathcal{L}(X_{0})$ is concentrated at a point, the initial $L^{2}$ distance to stationarity is not finite.

Later work has focused on stronger metrics for convergence or more demanding versions of the problem. In \cite{carlen2008entropy}, a very strong convergence condition as measured by entropy was discovered. These bounds, like the bounds relating to spectral gap, only imply convergence for sufficiently smooth initial distribution $\mathcal{L}(X_{0})$.
In this note, we focus on convergence bounds that do not depend on the initial distribution $\mathcal{L}(X_{0})$. The first bound with this property was obtained \cite{pak2007convergence}, in which the authors showed a convergence time of order $O(n^{2.5})$ in the $L^{1}$ Wasserstein metric, and \cite{Oliv07} improved this bound to $O(n^{2} \log(n))$ in the stronger $L^{2}$ Wasserstein metric. 
This latter bound is tight up to factors of $\log(n)$, and will be essential to our effort. Related Wasserstein bounds have also been found in \cite{cortez2014quantitative} for several similar models. However, a mixing bound in the Wasserstein metric does not directly imply any mixing bound in the total variation metric. \par

Thus the bounds obtained in \cite{pak2007convergence,Oliv07}, despite their strength, do not give  any information at all about the mixing time in total variation distance. The first bound on convergence in total variation was on the order of $4^{n^{2}}$ steps, obtained by Diaconis and Saloff-Coste in \cite{diaconis2000}. No progress was made on this problem until the recent unpublished work of Yunjiang Jiang \cite{jiangSOpre}, in which the author obtained a  mixing bound of order $n^{5} \log^{2}(n)$.

Theorem \ref{MainKacThm} of our paper also implies a mixing bound on the order of $n^{4 } \log(n)$ for Kac's walk on $\sn$. This improves upon all bounds prior to the present author's recent work \cite{pillai2015kac}, which shows matching upper and lower bounds on the order of $n \log(n)$ for this walk. The papers \cite{rosenthal1994random, porod1996cut, hough2012asymptotic,mischler2013kac,hauray2014kac} all study variants or projections of Kac's walk on $\so$. 
\subsection{Our Contributions} We have three main contributions: an order of magnitude improvement on the previous best bound for the convergence rate of Kac's walk in the strong total variation metric, new bounds on the smallest singular values of certain random matrices with dependent entries, and a general approach to bounding the mixing times of Gibbs samplers on spaces that are not `rectangular.' Our method also gives a way to compute effective mixing bounds via simulation. We now give a broad overview of our approach, and its relationship to some previous work.

The upper bound of $O(n^{4} \log(n))$ on the mixing time of Kac's walk is proved by using the popular \textit{coupling} technique: we run two copies $\{X_{t}\}_{t \geq 0}$, $\{ Y_{t} \}_{t \geq 0}$ of Kac's walk, and study the first time $\inf \{ t \geq 0 \, : \, X_{t} = Y_{t}\}$ that they collide. Like many non-Markovian couplings (see, \textit{e.g.}, \cite{hayes2003non,smith2014gibbs,czumaj2014thorp,pillai2015kac}), the main idea is to construct a coupling in two passes: an initial, Markovian coupling of `most' of the randomness in the chain, followed by a very general coupling of some `leftover' randomness.  Our initial coupling is exactly the one constructed in \cite{Oliv07}. Under this coupling, two copies of Kac's walk mix in Wasserstein distance after after $O(n^{2} \log(n))$ steps. Our contribution is in the construction and analysis of the second stage coupling. 

The usual approach to converting a Wasserstein mixing bound for a high-dimensional Gibbs sampler to a total variation mixing bound is via a greedy coupling: one attempts to match more and more coordinates as time progresses. Unfortunately, this approach works poorly for Kac's walk on $\so$. Indeed, as the authors discuss in \cite{pillai2015kac}, the greedy approach does not even work in the simpler case of Kac's walk on the sphere. Instead, we first run a scaffold $(\hat{X}_t,\hat{Y}_t)$ according to a coupling from \cite{Oliv07}. We then construct an $N$-dimensional perturbation of this scaffold by adding a small amount of additional randomness at each of $N$ time steps. The key point here is that it turns out to be easier to analyze the coupling probability of our two chains by using the $N$ bits of randomness all at once, rather than analyzing the coupling probability of the individual coordinate at each step as done in \cite{pillai2015kac}. Our approach is somewhat reminiscent of the `sprinkling' strategy used in random graph theory \cite{ajtai1982largest}, which also involves coupling `most' random variables in an intuitive way and then carefully analyzing a small amount of `leftover' randomness. 

Our analysis does not depend critically on any special properties of $\so$, and our non-Markovian coupling construction in Section \ref{DefCoupLongDesc} makes sense for any Gibbs sampler. The analysis of our coupling is also applicable for other constrained high dimensional Gibbs samplers. This is in stark contrast to our previous work \cite{pillai2015kac} on Kac's walk on the sphere.

Our simple results on random matrices in Section \ref{SecRandomMatrix} are novel, giving bounds on the smallest singular values of random matrices with significant dependence between entries. These bounds are closely related to the results of \cite{friedland2013simple,farrell2015smoothed} on the smallest singular values of random matrices with independent entries, and give bounds that are qualitatively similar to \cite{friedland2013simple}. 

In addition to giving asymptotic results on mixing times, our method allows us to numerically estimate the mixing time of Kac's walk for fixed $n$ by simulating the eigenvalues of a certain random matrix. This is useful for those interested in knowing the mixing time of a particular Gibbs sampler, and will generally give sharper results than our mathematical analysis. Note that estimating the mixing time of a Markov chain in this way is not trivial - \textit{a priori} it is not obvious how to obtain \textit{any} rigorous bounds on the mixing time of Kac's walk by a finite computation. See, \textit{e.g.}, \cite{conger2006shuffling} for further discussion on simulated estimates of mixing times.

\subsection{Notation}

For $S$ either a finite set or a subset of $\mathbb{R}^{m}$ with finite nonzero Lebesgue measure, we write $\unif(S)$ for the uniform probability measure on $S$. For functions $f,g: \mathbb{N} \mapsto \mathbb{R}$, we write $f=O(g)$ if $\limsup_{n \rightarrow \infty} \frac{|f(n)|}{|g(n)|} < \infty$. We also write $f = \Omega(g)$ if $g = O(f)$ and we write $f = \Theta(g)$ if both $f = O(g)$ and $f = \Omega(g)$. Finally, we write $f = o(g)$ if $\limsup_{n \rightarrow \infty} \frac{|f(n)|}{|g(n)|}= 0$. Unless otherwise noted, the terms inside of such `big-O' notation should always be taken with respect to the problem dimension $n$.

For $x,y \in \mathbb{R}^{m}$, we write $\langle x, y \rangle$ for the Euclidean inner product and $\| x - y \|$ for the associated norm. We also use `$0$' as a shorthand for the vector $(0,0,\ldots,0) \in \mathbb{R}^{m}$. Denote by $M(n)$ the collection of $n \times n$ matrices with real valued entries. For $h \in M(n)$, let $h^{\dag}$ denote its transpose. For a linear map $T: \mathbb{R}^{k} \mapsto \mathbb{R}^{\ell}$, define the operator norm:
\be
\| T \|_{\Op} =  \sup_{\|v \| = 1} \| T v\|.
\ee

For elements $x_{1}, x_{2}, \ldots$ of a noncommutative group, we use the convention:
\be 
\prod_{s=1}^{t} x_{s} \equiv x_{1} x_{2} \ldots x_{t-1} x_{t}.
\ee

We write $\partial S$ for the boundary of any set $S$. For any smooth manifold $\mathcal{M}$ and any $x \in \mathcal{M}$, we denote by $T_{x} \mathcal{M}$ the tangent space of $\mathcal{M}$ at $x$. 
We will often identify $T_{x} \mathbb{R}^{m}$ with $\mathbb{R}^{m}$. For any pair of smooth manifolds $\mathcal{M}$, $\mathcal{N}$ and any smooth function $f: \mathcal{M} \mapsto \mathcal{N}$, we define for $p \in \mathcal{M}$ the usual associated derivative map $\d f_{p}: T_{p} \mathcal{M} \rightarrow T_{f(p)} \mathcal{N}$. We recall an explicit construction of this map. Fix $v \in T_{p} \mathcal{M}$ and let $\gamma: [0,1] \rightarrow \mathcal{M}$ satisfy $\gamma(0) = p$, $\gamma'(0) = v$. Then 
\be 
\d f_{p}(v) = (f \circ \gamma)' (0).
\ee 
The quantity $\d f_{p}(v)$  is independent of the path $\gamma$ as long as $\gamma'(0) = v$. The rank of the linear map $\d f_{p}$ is denoted by $\text{Rank}(\d f_{p})$.

Let $G$ be a  Lie group $G$ with Lie algebra $\G$. For $a \in G$, let $\L_{a}: G \mapsto G$ be the left multiplication map 
\be 
\L_{a}(g) = ag.
\ee 
The exponential map, denoted by $\exp$, maps  $\G$ onto $G$. When $G$ is a matrix group and $\G$ is identified with a subset of $M(n)$, the exponential map  has the explicit form 
\be \label{eqn:expmaprep}
\exp(A) = \sum_{i=0}^{\infty} \frac{A^{i}}{i!}, \quad A \in \G
\ee 
(see Section 1.1 of \cite{abbaspour2007basic}). The Hilbert-Schmidt inner product on  $M(n)$ is: 
\be
\langle A, B \rangle_{\hs} = \Tr(A^{\dag}B) 
\ee
where $\Tr$ is the trace. The corresponding inner product is denoted by $\langle \cdot, \cdot \ra_{\hs}$. \par

 Throughout the paper, we will use the convention that the addition of angles $\theta, \theta'$ is always done modulo $2 \pi$, and that the distance between two angles is measured with respect to the usual metric on the torus rather than the usual metric on the line. \par

The Lie algebra $\G = \lag$ of $\so$ is the set of $n\times n$ skew-symmetric matrices
\be 
\lag = \{ h \in M(n) \, : \, h = - h^{\dag} \}.
\ee 
We denote by $D_{\hs}$ the Riemannian metric of $\so$ induced by the inner product $\la \cdot, \cdot \ra_{\hs}$ on $\lag$. The Haar measure $\pi$ on $\so$ is also induced by the inner product $\la \cdot, \cdot \ra_{\hs}$. We denote by $\mathcal{P}: \so \mapsto \lag$ the orthogonal projection operator into $\lag$ with respect to the Hilbert-Schmidt norm, \textit{i.e.},
\be
\mathcal{P}(g) = \mathrm{arg \,min}_{A \in \lag} \|g - A\|_{\hs}.
\ee
We construct an orthonormal basis for $\lag$ as follows. For $1 \leq i \leq N$, $\R_{i}(x) \equiv \R(i, x)$ (see Equation \ref{KacRep}) is a map from the $N$-dimensional torus $\mathbb{T} = [0, 2\pi)^{N}$ to $\so$. Set 
\be
a_{i} = \frac{1}{\sqrt{2}} \d\R_{i}(0) \in \lag.
\ee 
The set $\{a_i\}_{1 \leq i \leq N}$  constitute an orthonormal basis in $\lag$. This set of basis vectors were used in \cite{Oliv07} to obtain a coupling of two copies of random Kac's random walk on $\so$.

 Throughout the paper, the quantities $\phi_{n}$, $\epsilon_{n}$ and $\omega_{n}$ will control the three distance scales that are key to our coupling proof. The quantity $ \phi_{n}$, defined in Equation \eqref{DefScalingSingVal}, will satisfy $\phi_{n} \leq (2n)^{-20}$ and controls the scale on which a certain function `looks flat.' The quantity $\epsilon_{n} \leq \phi_{n}^{30}$ will control the total amount of `injected randomness' available to our coupling. Finally, $\omega_{n} = \epsilon_{n}^{30}$ controls the typical distance between the two Markov chains that we are trying to couple. 

\section{Proof Strategy} \label{SecBigPictureProof}
We give an overview of the proof of Theorem \ref{MainKacThm}. A basic ingredient in the proof of the upper bound on the mixing time in Theorem \ref{MainKacThm} is the following standard coupling inequality:
\begin{lemma}\label{LemmaIneqCoup}
Let $K$ be the transition kernel of a Markov chain with unique stationary distribution $\nu$ on state space $\Omega$. Let $\{X_{t} \}_{t \geq 0}$, $\{Y_{t} \}_{t \geq 0}$ be two Markov chains, with $Y_{0} \sim \nu$ started at stationarity and $X_{0} = x \in \Omega$. Define the coalescence time 
\be \label{DefCoalTime}
\tau(x) = \inf \{ t \, : \, X_{t} = Y_{t} \}.
\ee Assume that the coupling of $\{X_{t} \}_{t \geq 0}$, $\{Y_{t} \}_{t \geq 0}$ satisfies $ X_{t} = Y_{t}$ for all $t \geq \tau(x)$. Then 
\be 
\| \mathcal{L}(X_{t}) - \nu \|_{\TV} \leq \P[\tau(x) > t].
\ee 
\end{lemma}
\vspace{0.2in}

Let $\{X_t\}_{t \geq 0},$ $\{Y_{t}\}_{t \geq 0}$ be two copies of Kac's walk as described in Equation \eqref{KacRep}. Let $\{i_{t}(x),\theta_{t}(x)\}_{t \geq 0}$ and $\{i_{t}(y),\theta_{t}(y)\}_{t \geq 0}$ be the random variables that determine the evolution of $\{X_{t}\}_{t \geq 0}$ and $\{Y_{t} \}_{t \geq 0}$ in Equation \eqref{KacRep}; we refer to such sequences as \textit{update sequences}. To construct a coupling of $\{ X_{t} \}_{t \geq 0}$ and $\{ Y_{t} \}_{t \geq 0}$, it is enough to construct a coupling of $\{i_{t}(x),\theta_{t}(x)\}_{t \geq 0}$ and $\{i_{t}(y),\theta_{t}(y)\}_{t \geq 0}$.

The following coupling was defined and analyzed in \cite{Oliv07}.

\begin{defn} [Locally Contractive Coupling \cite{Oliv07}] \label{DefContCoup}
For any fixed $T \in \mathbb{N}$, the following algorithm gives a coupling of $\{X_{t}, Y_{t}\}_{t = 0}^{T}$ according to their associated update sequences:
\begin{enumerate}
\item For $0 \leq t < T$, choose $i_{t} \sim \unif \{1,2,\ldots,N\}$ and set 
\be 
i_{t}(y) = i_{t}(x) = i_{t}.
\ee 
\item  Let $\{ \eta_{t}(x) \}_{0 \leq t < T}$ be an i.i.d. sequence of $\unif[0,2 \pi)$ random variables. We now construct the sequences $\{ X_{t}, Y_{t}, h_{t}, \eta_{t}(y) \}_{t \geq 0}$ inductively in $t$. For $1 \leq s \leq t$, define 
\be X_{s} = \Big( \prod_{u=s-1}^{0} \R(i_{u}(x), \eta_{u}(x))\Big)X_{0}, \,\, Y_{s} = \Big( \prod_{u=s-1}^{0} \R(i_{u}(y), \eta_{u}(y))\Big)Y_{0}.
\ee
Finally, set  
\be \label{EqNonMarkovSoNotSn}
h_{s} &= \mathcal{P}(Y_{s} X_{s}^{\dag} - \mathrm{Id}) \\
\eta_{s}(y) &= \eta_{s}(x) + {1 \over \sqrt{2}} \la h_{s}, a_{i_{s}} \ra_{\hs}.
\ee 
\item Set $\theta_{t}(x) = \eta_{t}(x)$, $\theta_{t}(y) = \eta_{t}(y)$ for all $0 \leq t < T$.
\end{enumerate}
\end{defn}

Theorem 1 of  \cite{Oliv07} implies that it is possible to couple\footnote{The coupling referred to in Theorem 1 of \cite{Oliv07} is not the same as the coupling given in Definition \ref{DefContCoup}, though the two are closely related. } two copies of Kac's random walk
$\{X_{t}\}_{t \geq 0}$ and  $\{Y_{t} \}_{t \geq 0}$  so that, for any $0 < A < \infty$ and some sequence $T = T(n) = O(A \, n^{2} \log(n))$, they satisfy
\be
\sup_{ (X_{0}, Y_{0}) \in \so \times \so} \P[ \|X_{T} - Y_{T} \|_{\hs} \leq n^{-A}] \geq 1 - n^{-A}, \quad \quad T = O(A \, n^2 \log n).
\ee

The following result, obtained as a consequence of this result, will allow us to restrict our attention to starting points $X_{0}$ that satisfy $\|X_{0} - Y_{0}\|_{\hs} \leq \omega_{n}$ without loss of generality:

\begin{lemma} \label{LemmaOlivBound}
Denote by $\Pi$ the class of measures $\nu$ on $\so\times\so$ that are supported on pairs $(x,y)$ with $\|x-y\|_{\hs} \leq \omega_{n}$. Assume that, for some $T \in \mathbb{N}$ and some $0 < c < 1$, 
\be \label{IneqOlivBoundMainAssumption}
\sup_{ (X_{0}, Y_{0}) \sim \nu \in \Pi}  \| \mathcal{L}(X_{T}) - \mathcal{L}(Y_{T}) \|_{\TV}  \leq c.
\ee 
Then for any $A > 0$, $C > \frac{\log(\omega_{n})}{\log(n)} + A$ and $S > \lceil n^{2} \log(n) \left( \frac{1}{2} + \log(\pi) + C  \right) \rceil$,
\be \label{IneqOlivBoundMainConclusion}
\sup_{X_{0} = x \in \so} \| \mathcal{L}(X_{T + S}) - \mu \|_{\TV} \leq c + n^{-A}.
\ee 
\end{lemma}
\begin{proof}
This is an immediate consequence of Theorem 1 of \cite{Oliv07}. For any $X_{0} = x \in \so$ and $Y_{0} \sim \mu$, Theorem 1 of \cite{Oliv07} implies that there is a coupling of $\{(X_{t}, Y_{t}) \}_{t=0}^{S}$ so that 
\be 
\E[\|X_{S} - Y_{S} \|_{\hs}^{2}] \leq \omega_{n} n^{-A}. 
\ee 
By Markov's inequality, it follows that for this coupling, we can write the joint law of $(X_S, Y_S)$ as:
\be \label{IneqUsingOlivIntMinCond}
\mathcal{L}(X_{S},Y_{S}) = (1 - n^{-A}) \nu_{x} + (n^{-A}) r_{x}
\ee 
for some $\nu_{x} \in \Pi$ and some probability measure $r_{x}$ on $\so \times \so$; we do not need to know the details of either of these distributions. 

Now the claim follows from a standard minorization argument. Let $X_{0} = x \in \so$, $Y_{0} \sim \mu$ and $Z$ be a Bernoulli random variable with $\P[Z = 0] = 1 - \P[Z = 1] = n^{-A}$. We couple $\{ X_{t} \}_{t=0}^{S}$, $\{ Y_{t} \}_{t=0}^{S}$ and $Z$ so that
\be 
\mathcal{L}(X_{S},Y_{S}) = \Big\{
		\begin{array}{ll}
			\nu_{x} & \mbox{if } Z = 1 \\
			r_{x} & \mbox{if } Z = 0.
		\end{array}
\ee 
This coupling is possible by 
Equation \eqref{IneqUsingOlivIntMinCond}. If $Z=1$, we then couple $\{X_{t}\}_{t=S}^{T+S}$, $\{Y_{t}\}_{t=S}^{T+S}$ so that 
\be 
\P[X_{T+S} \neq Y_{T+S}] \leq c;
\ee 
such a coupling is possible by our hypothesis stated in Equation \eqref{IneqOlivBoundMainAssumption}. If $Z=1$, we allow $\{ X_{t}\}_{t \geq S}$, $\{ Y_{t}\}_{t \geq S}$ to evolve independently. Under this coupling, we have 
\be 
\P[X_{T+S} \neq Y_{T+S}] &= \P[X_{T+S} \neq Y_{T+S} | Z=0] \P[Z=0] + \P[X_{T+S} \neq Y_{T+S} | Z=1] \P[Z=1] \\
&\leq \P[Z=0] + \P[X_{T+S} \neq Y_{T+S} | Z=1] \\
&\leq n^{-A} + c.
\ee 
By Lemma \ref{LemmaIneqCoup}, this completes the proof.
\end{proof}

Now we define our second stage coupling in detail. 
We construct our coupling by running two `scaffolding' copies of Kac's walk (these are the chains $\{ \hat{X}_{t}, \hat{Y}_{t} \}_{t \geq 0}$ defined below) according to Definition \ref{DefContCoup}, and then insert a little bit of extra randomness at carefully chosen points (this corresponds to the choice of $\delta(x)$, $\delta(y)$ below). We first need the following definitions.

\begin{defn}[Induced Map] \label{EqIndMap}
Fix  $\mathcal{T} \in \mathbb{N}$ and a set of integers $\mathcal{S} = \{s_{1} < \ldots < s_N\}$ with $0 \leq s_{1} < s_N < \mathcal{T}$ . Fix a sequence $\mathcal{I} = \{i_{0},\ldots, i_{\mathcal{T}-1} \}$  with $i_k \in  \{1, 2, \cdots, N\}$ and a sequence $\eta = \{{\eta}_{0},\ldots, {\eta}_{\mathcal{T}-1} \}$ with $\eta_k \in [0,2\pi)$. Fix $X \in \so$ and let $\mathcal{A} = \{X,\mathcal{T}, \mathcal{S}, \mathcal{I},  \eta \}$. \par
Define the maps $e_{\mathcal{A},t} : [0,2 \pi)^{N} \rightarrow [0, 2 \pi)$ by
\be 
e_{\mathcal{A},t}(x_{1},\ldots,x_{N}) &= {\eta}_{t}, \, \, \quad \qquad  \qquad t \notin \mathcal{S} \\
e_{\mathcal{A},t}(x_{1},\ldots,x_{N}) &= {\eta}_{t} + x_{\ell}, \, \quad \qquad t = t_{\ell} \in \mathcal{S}. 
\ee 
\noindent Finally, fix $0 < c < \pi$ and define the map $f_{\mathcal{A},c}:[-c,c]^{N} \rightarrow \so$ by
\be \label{EqDefMainMap}
f_{\mathcal{A},c}(x_{1},\ldots,x_{N}) =  \prod_{t=\mathcal{T}-1}^{0} \R(i_{t}, e_{\mathcal{A},t}(x_{1},\ldots,x_{N})) X.
\ee 
\end{defn}

Our coupling will involve the choice of certain marked times $s_{1} < s_{2} < \ldots < s_{N}$. The following `greedy' strategy for choosing these marked times is difficult to analyze, but gives a useful algorithm for obtaining mixing bounds via simulation (see Remark \ref{SecComp}):

\begin{defn} [Greedy Subset Choice] \label{DefSubsetChoiceGreedy}
Let $\{i(t)\}_{t \geq 0 }$ be a sequence with $i(t) \in \{1,2,\ldots,N\}$. For $t \geq 0$, define $V_{t} = \mathrm{span}( \{a_{i(1)}, a_{i(2) }, \ldots, a_{i(t)} \} )$. Then define $s_{1} = 0$ and inductively define 
\be 
s_{\ell + 1} = \inf \{s > s_{\ell} \, : \, V_{s} \neq V_{s_{\ell}} \}.
\ee  Let $T = s_{N}+1$ and $\mathcal{S} = \{s_{1},\ldots,s_{N}\}$.
Note that  the quantities $\mathcal{S}$ and $T$ are both functions only of the sequence $\{i_{t}\}_{t \geq 0}$. 
\end{defn}

A critical step in the analysis of our Markov chain is a bound on the smallest singular value of a matrix $D$ with highly dependent entries that is associated with our subset choice (see Equation \eqref{EqJacMatCompRedux}). Since $D$ has highly dependent entries, we are only be able to bound its smallest singular value by comparing $D$ to a `limiting' random matrix $D_{\infty}$ for which we can make exact computations (see Section \ref{SecRandomMatrix} and Remark \ref{RemWhyDInf}). To make this comparison possible, $D$ must be close to $D_{\infty}$. The following `lazy' choice of marked times ensures that this happens: 

\begin{defn} [Lazy Subset Choice] \label{DefSubsetChoice}
Fix a constant $\mathcal{Q} > 0$. Let $\{i(t)\}_{t \geq 0 }$ be a sequence with $i(t) \in \{1,2,\ldots,N\}$. We inductively define the sequence $\{ s_{\ell}\}_{\ell=1}^{N}$ by setting 
\be 
s_{1} &= \min \{ t \geq 0\, : \, i(t) = 1\} \\
s_{\ell + 1} &= \min \{ t \geq s_{\ell} + \mathcal{Q} n^{2} \log(n) \, : \, i(t) = \ell + 1\}.
\ee 

Define 
\be \label{EqDimensionIncreasingTimesGreedy}
\mathcal{S} = \{s_{1}, \ldots, s_{N} \}
\ee 
and set $T = s_{N}+1$.
\end{defn}

\subsection{Non-Markovian Coupling} 
\label{DefCoupLongDesc}

We now give a non-Markovian coupling for two copies $\{X_{t}\}_{t\geq 0}$, $\{ Y_{t} \}_{t \geq 0}$ of Kac's walk. Although this coupling will  be used in the proof of Theorem \ref{MainKacThm} for two chains that are started at nearby points, the coupling itself makes sense for any pair of starting points $(X_0,Y_0)$. The coupling depends on the parameters $\mathcal{Q} \in \mathbb{R}^{+}$ and $\epsilon_{n} \in (0,1)$.

As with the coupling in Definition \ref{DefContCoup}, we define this coupling in terms of the associated update sequences $\{i_{t}(x),\theta_{t}(x)\}_{t \geq 0}$ and $\{i_{t}(y),\theta_{t}(y)\}_{t \geq 0}$:

\begin{enumerate}
\item For $t \geq 0$, choose $i_{t} \sim \unif \{1,2,\ldots,N\}$ and set 
\be \label{eqn:ixycoup}
i_{t}(y) = i_{t}(x) = i_{t}.
\ee 
\item Choose $\mathcal{S}$ and $T$ as in Definition \ref{DefSubsetChoice}.
\item  We couple two copies $\{\hat{X}_{t}, \hat{Y}_{t}\}_{t \geq 0}$ of Kac's walk started at $\hat{X}_{0} = X_{0}$, $\hat{Y}_{0} = Y_{0}$ according to Definition \ref{DefContCoup} and using the same sequence $\{i_{t}(x), i_{t}(y)\}_{t \geq 0}$; let  $\{ \eta_{t}(x), \eta_{t}(y), h_{t} \}_{0 \leq t < T}$ be as in Definition \ref{DefContCoup}. These walks `scaffold' our coupling of $\{X_{t}, Y_{t}\}_{t \geq 0}$.

\item   Define 
\be \label{eqn:sAB}
\mathcal{A}  &= \{X_{0}, T, \mathcal{S}, \{i_{0},\ldots,i_{T-1} \}, \{ \eta_{0}(x),\ldots,\eta_{T-1}(x) \} \}, \\
\mathcal{B} &= \{Y_{0}, T, \mathcal{S}, \{i_{0},\ldots,i_{T-1} \}, \{ \eta_{0}(y),\ldots,\eta_{T-1}(y) \} \}.
\ee
Let $\{\delta_{t}(x)\}_{1 \leq t \leq N}$, $\{\delta_{t}(y)\}_{1 \leq t \leq N}$ be two sequences of i.i.d. $\unif(-\epsilon_{n},\epsilon_{n})$ random variables and write $\delta(z) = (\delta_{1}(z),\ldots,\delta_{N}(z))$ for $z \in \{x,y\}$. We couple the two sequences $\delta(x)$, $\delta(y)$ so that they satisfy 
\be \label{EqWantSmall} 
\P[f_{\mathcal{A}, \epsilon_{n}}(\delta(x)) = f_{\mathcal{B}, \epsilon_{n}}(\delta(y))] = \| \mathcal{L}(f_{\mathcal{A}, \epsilon_{n}}(\delta(x))) -  \mathcal{L}(f_{\mathcal{B}, \epsilon_{n}}(\delta(y))) \|_{\TV}.
\ee 
For $t < T$, set 
\be [eqn:thetaxcoup] 
\theta_{t}(x) &= \eta_{t}(x), \, \, \quad \qquad  \qquad t \notin \mathcal{S} \\
\theta_{t}(x) &= \eta_{t}(x) + \delta_{\ell}(x), \qquad t = t_{\ell} \in \mathcal{S}, 
\ee 
and similarly
\be [eqn:thetaycoup]
\theta_{t}(y) &= \eta_{t}(y), \, \, \quad \qquad  \qquad t \notin \mathcal{S} \\
\theta_{t}(y) &= \eta_{t}(y) + \delta_{\ell}(y), \qquad t = t_{\ell} \in \mathcal{S}. 
\ee 
This completes our construction of $\theta_{t}(x)$, $\theta_{t}(y)$ for $0 \leq t < T$.
\item  To construct $\theta_{t}(x)$, $\theta_{t}(y)$ for $t \geq T$, let $\{ \theta_{t} \}_{t \geq T}$ be a sequence of i.i.d. random variables with $\theta_{t} \sim \unif [0,2\pi)$ and set $\theta_{t}(x) = \theta_{t}(y) = \theta_{t}$. 
\end{enumerate}
Thus Equation \eqref{eqn:ixycoup} couples $i_t(x)$ and $i_t(y)$ and Equations \eqref{eqn:thetaxcoup} and \eqref{eqn:thetaycoup} give the coupling of $\theta_t(x)$ and $\theta_t(y)$. The stochastic processes $\{X_t\}$ and $\{Y_t\}$ are coupled through this non-Markovian coupling of update sequences.

\begin{lemma} [Validity of Coupling]
The stochastic process $\{X_{t} \}_{t \geq 0}$, $\{ Y_{t} \}_{t \geq 0}$ have the same distribution as Kac's walk.
\end{lemma}
\begin{proof}
We give the proof for $\{ X_{t} \}_{t \geq 0}$, as the proof for $\{ Y_{t} \}_{t \geq 0}$ is identical. To check that $\{ X_{t} \}_{t \geq 0}$ has the same distribution as Kac's walk, it is sufficient to check that the update sequence $\{ (i_{t}(x), \theta_{t}(x)) \}_{t \geq 0}$ has the correct distribution. The sequence $\{ i_{t}(x) \}_{t \geq 0}$ is defined to be an i.i.d. sequence of $\unif( \{1,2,\ldots,N\} )$ random variables, so this is correct. 

To check that $\{ \theta_{t}(x) \}_{t \geq 0}$ is an i.i.d. sequence of $\unif[0, 2 \pi)$ random variables conditional on $\{ i_{t}(x) \}_{t \geq 0}$, we note that conditional on $\{\delta_{t}(x) \}_{t=1}^{N}$ and  $\{ i_{t}(x) \}_{t \geq 0}$, $\{ \eta_{t}(x) \}_{t \geq 0}$ is an i.i.d. sequence of $\unif[0, 2 \pi)$ random variables. Since any random variable plus an independent $\unif[0, 2\pi)$ random variable is itself uniform on $[0,2\pi)$, this implies $\{ \theta_{t}(x) \}_{t \geq 0}$ is an i.i.d. sequence of $\unif[0, 2 \pi)$ random variables.
\end{proof}

\begin{remark} \label{RemarkCoupGenerality}
The non-Markovian coupling in Section \ref{DefCoupLongDesc} makes sense for any Gibbs sampler. To use it, one only requires
\begin{enumerate}
\item a representation of the Gibbs sampler in terms of the sequence of randomly-chosen coordinates to be updated $\{i_{t}\}_{t \geq 0}$ and the random quantile $\{ \theta_{t}\}_{t \geq 0}$, and
\item a candidate contractive coupling to form the `scaffold.'
\end{enumerate}
These requirements are not onerous; {(1)} is the usual way to write down a Gibbs sampler, while any optimal 1-step coupling is a reasonable candidate for {(2)}.
\end{remark}

\subsection{Informal Description of Coupling Analysis}\label{RemarkCoupGoal}
By construction, we have $X_{T} =f_{\mathcal{A}, \epsilon_{n}}(\delta(x))$, $Y_{T} = f_{\mathcal{B}, \epsilon_{n}}(\delta(y))$. Consequently, any bound on the total variation distance in Equation \eqref{EqWantSmall} immediately gives a bound on the coupling probability required by Lemma \ref{LemmaIneqCoup}. To prove a mixing bound, it is thus enough to prove that the images $I_{X} = f_{\mathcal{A},\epsilon_{n}}([-\epsilon_{n},\epsilon_{n}]^{N})$ and $I_{Y} = f_{\mathcal{B},\epsilon_{n}}([-\epsilon_{n},\epsilon_{n}]^{N})$ of $f_{\mathcal{A},\epsilon_{n}}$ and $f_{\mathcal{B},\epsilon_{n}}$ have a large overlap and the respective Jacobians $M_\mathcal{A}$ and $M_\mathcal{B}$ are roughly constant. By Taylor expansion, these images are roughly of the form 
\be
I_{X} \approx f_{\mathcal{A},\epsilon_{n}}(0) + \exp(M_{\mathcal{A}} [-\epsilon_{n},\epsilon_{n}]^{N}), \, I_{Y} \approx f_{\mathcal{B},\epsilon_{n}}(0) + \exp(M_{\mathcal{B}}[-\epsilon_{n},\epsilon_{n}]^{N}). 
\ee
By the assumption that $\|X_{0} - Y_{0} \| \leq \omega_{n}$ is small and the main result of \cite{Oliv07}, we have $\| f_{\mathcal{A},\epsilon_{n}}(0) - f_{\mathcal{B},\epsilon_{n}} \| = O(\omega_n)$ and $\| M_{\mathcal{A}} - M_{\mathcal{B}}\| = O(\omega_n)$. Thus, to check that $I_{X} \approx I_{Y}$, it is enough to check that the smallest singular value $\sigma_{1}(M_{\mathcal{A}})$ of $M_{\mathcal{A}}$ satisfies $\sigma_{1}(M_{\mathcal{A}})\gg \| M_{\mathcal{A}} - M_{\mathcal{B}}\|$.
Similar bounds show that the Jacobians are nearly constants. This heuristic is formalized in Section \ref{SecCouplingArgumentAlone}.

\section{Technical Lemmas} \label{SecBigMainLemmas}

We give a collection of general estimates that will be used throughout the paper. The proofs are deferred to Appendix \ref{SecAppProofsBoring}, and can be safely skipped on a first reading of this paper.

\subsection{Matrix Estimates}
We use the following result repeatedly.

\begin{lemma} \label{LemmaDifProdSwap}
Fix $k \geq 1$ and let $P_{1},\ldots, P_{k}$ and $Q_{1},\ldots, Q_{k}$ be two sequences of elements of $M(n)$. Then
\be \label{IneqProdMats}
\Big \| \prod_{i=1}^{k} Q_{i} - \prod_{i=1}^{k} P_{i}\Big \|_{\hs} \leq \sum_{i=1}^{k} \Big \| \prod_{\ell=1}^{i-1} Q_{\ell}\Big \|_{\mathrm{Op}} \,  \| Q_{i} - P_{i} \|_{\hs} \, \Big\|\prod_{\ell=i+1}^{k}P_{\ell} \Big\|_{\mathrm{Op}}.
\ee 
\end{lemma}

We also have the elementary bound:

\begin{lemma} [Determinant Estimate] \label{LemmaJacobianBound}
Let $M_{1}, M_{2}$ be two $N$ by $N$ symmetric matrices. For a general matrix $M$, denote by $\sigma_{1}(M) \leq \sigma_{2}(M) \leq \ldots \leq \sigma_{N}(M)$ the ordered singular values of $M$. Assume that 
\be \label{LemmaJacBoundMainReq}
\| M_{1} - M_{2} \|_{\mathrm{Op}} \leq \delta \sigma_{1}(M_{1})
\ee
for some $0 < \delta < 1$. Then the determinants of $M_{1}$, $M_{2}$ satisfy:
\be 
|\frac{\mathrm{det}(M_{2})}{\mathrm{det}(M_{1})} - 1|  \leq N^{\frac{N}{2}} \delta^{N}.
\ee 

\end{lemma}

Our next collection of bounds will require further notation. Define two functions  $f: [-c,c]^N \mapsto \so$ and
 $g: [-c,c]^N \mapsto \so$:
\be \label{GenFuncDef}
f(x_1,x_2, \cdots, x_N) &= \prod_{k =1}^N \R_k \,\exp{\big((\t_k + x_k) a_k\big)}, \\
g(y_1,y_2,\cdots, y_N) &= \prod_{k =1}^N \tR_k \,\exp{\big((\tt_k + y_k) a_k\big)},
\ee
where $\t_k, \tilde{\t}_k \in [0,2\pi)$, $\R_k, \tR_k \in \so$ and $a_k \in \lag$. For $1 \leq i + 1 < j \leq N$, define
\be 
M_{i,j} = M_{i,j}(x_{1},\ldots,x_{N}) \equiv \prod_{k =i+1}^{j-1}  \R_k \,e^{(\t_k + x_k) a_k},
\ee 
for $1 \leq j + 1 < i \leq N$, define 
\be 
M_{i,j} = \prod_{k=j+1}^{i-1} R_{k} \,e^{(\t_k + x_k) a_k},
\ee 
and for $1 \leq i < N$ define $M_{i,i+1} = M_{i+1,i} = \mathrm{Id}$. The derivative map $\d f : T_p[-c,c]^N \mapsto T_{f(p)} \so$ in the direction
$h = \sum_k h_k {\partial \over \partial e_k}$ is 
\be \label{EqGenFuncDer}
\d f_x(h) =  \sum_{j=1}^N \prod_{k =1}^N \R_k \,e^{(\t_k + x_k) a_k} (\I + (h_k a_k - \I)1_{k=j}).
\ee

We state two lemmas that bound various approximations related to $f$ and $g$:

\begin{lemma} [Closeness of Tangent Maps] \label{LemmaTangentClose}
Let $f: [-c,c]^N \mapsto \so$ be as defined in Equation \eqref{GenFuncDef}. Assume that $c < N^{-3}$ and $\max_{k} |\t_k - \tt_k | \leq c^{2}$, where $\t_k,\tt_k \in [0,2\pi)$ are as in Equation \eqref{GenFuncDef}. Then for all $x,y \in [-c,c]^N$ and all $h \in \mathbb{R}^{N}$ with $\| h \| \leq 1$,
\be \label{IneqTanCloseTwoConc}
\|  \d f_x(h) - \d f_y(h) \|_\hs &\leq  4 N^{2} c \\
\| \d\L_{f(y) (f(x))^{-1}} \d f_x(h) - \d f_y(h) \|_\hs &\leq 8 N^{2} c.
\ee
\end{lemma}

\begin{lemma} [Approximation by Exponential Map] \label{ClosenessOfExp}
Let $f$ be of the form given in Equation \eqref{GenFuncDef}. Then for $c < N^{-3}$,
\be 
\| f(x) - f(0) \,\exp(\d\L_{f(0)^{-1}} \d f_{0}(x)) \|_{\hs}^{2} \leq 8 N^{2} c^{2}.
\ee 
\end{lemma}

We need the following expression for the Jacobian of $f$ and $g$:

\begin{lemma} [Jacobian Formula] \label{LemmaTangentSize}
Let $f$ be the function defined in Equation \eqref{GenFuncDef}. For $1 \leq i < j \leq N$, define
\be \label{EqJacMatComp}
D_{i,j} = -\Tr[ a_i  M_{i,j} \R_{j}  a_{j}  \R_{j}^{-1} M_{i,j}^{-1} ];
\ee 
for $1 \leq j < i \leq N$, define 
\be 
D_{i,j} = D_{j,i}.
\ee 
Finally, set $D_{i,i} = 1$ for all $1 \leq i \leq N$ and let $D$ be the matrix with entries $[D_{i,j}]$.  Then for all $h = (h_{1},\ldots,h_{N}) \in \mathbb{R}^{N} \backslash \{ 0 \}$ and all $x \in [-c,c]^{N}$,
\be \label{ExactTangentCalc}
\langle \d f_{x}(h),\d f_{x}(h') \rangle_{\hs}  =  \langle h, h' \rangle + \sum_{i \neq j } h_{i} h_{j}' D_{i,j} =  h^{\dag} D h'.
\ee 
\end{lemma}

We point out that $f$ is generally a diffeomorphism for $c$ sufficiently small:

\begin{lemma} \label{LemmaDiff}
Let $f$ be of the form given in Equation \eqref{GenFuncDef} and assume that it satisfies 
\be \label{IneqTangentMapNotSmall}
\inf_{h \neq 0} \frac{  2 \| df_{x}(h) \|}{ \| h \|}  &\geq \phi_{n} \\
\sup_{h \neq 0} \frac{   \| df_{x}(h) \|}{ \| h \|}  &\leq N,
\ee  for some sequence 
\be \label{IneqSillyPhiCond}
\phi_{n} \leq (2n)^{-30}.
\ee 
Then there exist constants $0 < C_0 < 1$, $N_{0} < \infty$ that do not depend on $n$ or the particular sequence $\{ \phi_{n} \}_{ n \geq 0}$ satisfying \eqref{IneqSillyPhiCond}, such that the function $f$ is a diffeomorphism whenever $c < C_{0} n^{-6} \phi_{n}$ for all $n > N_{0}$ sufficiently large. 
\end{lemma}

\subsection{Probability Estimates}

We check that the timescale $T$ of the non-Markovian coupling defined in Section \ref{DefCoupLongDesc} is not too large:

\begin{lemma} \label{LemmaDimensionTimeGreedy}
Let $s_{N}$ be defined as in Definition \ref{DefSubsetChoiceGreedy}.
For all $c \in \mathbb{R}$,
\be \label{EqCoupConLim}
\lim_{n \rightarrow \infty} \P[s_{N} <  N \log(N) + cN] = e^{-e^{c}}.
\ee 
\end{lemma}

\begin{lemma} \label{LemmaDimensionTime}
Let $s_{N}$ be defined as in Definition \ref{DefSubsetChoice}. Then for all $k \geq 2$,
\be \label{EqCoupConLim}
 \P[s_{N} > \mathcal{Q} N^{2} \lceil \log(N) \rceil + k N^{2}] \leq e^{-\frac{N}{4}}.
\ee 
\end{lemma}

We check that the chains in Section \ref{DefCoupLongDesc} will stay very close to each other, with high probability:

\begin{lemma} [Closeness of Paths] \label{LemmaPathCloseness}

Assume that $X_{0}$, $Y_{0}$ satisfy $\| X_{0} - Y_{0} \|_{\mathrm{HS}} \leq \omega_{n}$ for some $0< \omega_{n} \leq n^{-2700}$. Let $\{X_{t},Y_{t}, \hat{X}_{t}, \hat{Y}_{t}\}_{t \geq 0}$ be coupled as in Definition \ref{DefCoupLongDesc}, with $c \equiv \epsilon_{n} \in [\omega_{n}^{\frac{1}{30}}, n^{-900}]$. Fix $0 < C < \frac{\log(\omega_{n})}{\log(n)} - 4$. Then 
\be \label{IneqPathClosenessConc1}
\P[\sup_{0 \leq t \leq n^{5}} \| \hat{X}_{t} - \hat{Y}_{t} \|_{\hs} \leq 2 \| \hat{X}_{0} - \hat{Y}_{0} \|_{\hs} n^{5 + C} ]  \geq 1 - n^{-C} 
\ee 
\be \label{IneqPathClosenessConc2}
\sup_{0 \leq t \leq n^{5}} \| X_{t} - \hat{X}_{t} \|_{\hs} \leq 6 n^{5} \epsilon_{n} 
\ee 
\be \label{IneqPathClosenessConc3}
\sup_{0 \leq t \leq n^{5}} \| Y_{t} - \hat{Y}_{t} \|_{\hs} \leq 6 n^{5} \epsilon_{n}. 
\ee 
\end{lemma}

\section{Coupling Argument} \label{SecCouplingArgumentAlone}

We now give a generic bound on the total variation distance in Equation \eqref{EqWantSmall}, under the assumption that nothing `too bad' happened. This formalizes the argument in Section \ref{RemarkCoupGoal}.

\begin{lemma} [Coupling Construction] \label{LemmaCoupConAbs}
Fix sequences $\phi_{n} \leq (2n)^{-30}$, $\epsilon_{n} = \phi_{n}^{30}$ and $\omega_{n} = \epsilon_{n}^{30}$. Fix a constant $C$ that satisfies $0 < C < \frac{\log(\omega_{n})}{\log(\epsilon_{n})} - 6$ for all $n \in \mathbb{N}$, and let $f = f_{\mathcal{A}, \epsilon_{n}}$ and $g = f_{\mathcal{B}, \epsilon_{n}}$ be as defined in Equation \eqref{EqDefMainMap}, with $\mathcal{A}$, $\mathcal{B}$ be as in equation \eqref{eqn:sAB}. Assume that
\be \label{IneqCoupCond1}
\| f(0) - g(0) \|_{\hs} \leq \omega_{n} n^{5 + C} 
\ee 
and that for $q \in \{ f,g \}$, we have: 
\be \label{IneqTangentMapNotSmallPartTwo}
\inf_{h \neq 0} \frac{  2 \| dq_{x}(h) \|}{ \| h \|}  &\geq \phi_{n} \\
\sup_{h \neq 0} \frac{   \| dq_{x}(h) \|}{ \| h \|}  &\leq N.
\ee 

Then it is possible to couple two sequences of i.i.d. random variables $U_{1},\ldots,U_{N} \sim \unif(-\epsilon_{n},\epsilon_{n})$ and $V_{1},\ldots,V_{N} \sim \unif(-\epsilon_{n},\epsilon_{n})$ so that
\be
\P[f(U_{1},\ldots,U_{N}) \neq g(V_{1},\ldots,V_{N})] \leq 513 N^{2} \phi_{n}^{-1}  \epsilon_{n} = o(1).
\ee
\end{lemma}
\begin{proof}

By Lemma \ref{ClosenessOfExp},
\be \label{IneqCoupCompToExp}
\| f(x) - f(0) \,\exp( \d\L_{f(0)^{-1}} \circ \d f_{0}(x)) \|_{\hs} &\leq 8 N^{2} \epsilon_{N}^{2} \\
\| g(x) - g(0)\, \exp( \d\L_{g(0)^{-1}} \circ \d g_{0}(x)) \|_{\hs} &\leq 8 N^{2} \epsilon_{N}^{2}.
\ee 
for all $x \in (-\epsilon_{n}, \epsilon_{n})^{N}$. For $0<r<\infty$, define 
\be 
\mathcal{H}_{f}(r) &= \{ f(0) \, \exp(\d \L_{f(0)^{-1}} \circ \d f_{0}(x)) \, : \, x \in [-r,r]^{N} \} \\
\mathcal{H}_{g}(r) &= \{ g(0) \, \exp( \d \L_{g(0)^{-1}} \circ \d g_{0}(x)) \, : \, x \in [-r,r]^{N} \}.
\ee

We claim: \begin{prop} \label{PropOriginalBoxes}
For $c > 0$, set 
\be \label{IneqOrigBoxCond}
u_{1} &= c - 32 N^{2} \phi_{n}^{-1}  \epsilon_{n}^{2} \\
u_{2} &= c + 32 N^{2} \phi_{n}^{-1}  \epsilon_{n}^{2}. 
\ee 
Then for all $n$ sufficiently large,
\be \label{IneqContainments}
\mathcal{H}_{f}(u_{1}) &\subset  f([-c,c]^{N}) \subset \mathcal{H}_{f}(u_{2}) \\
\mathcal{H}_{g}(u_{1}) &\subset  g([-c,c]^{N}) \subset \mathcal{H}_{g}(u_{2}),
\ee 
uniformly in sequences $c = c_{n}$ satisfying $\frac{1}{2} \epsilon_{n} \leq c \leq 2 \epsilon_{n}$.
\end{prop}
\begin{proof}
Assume without loss of generality that $f(0) = \mathrm{Id}$. Fix $x \in \partial[-c,c]^{N}$. Since the exponential map is surjective (see, \textit{e.g.}, Theorem 6.9.3 of \cite{abbaspour2007basic}), we can write $f(x) = \exp(h)$. By taking a solution $h$ to $f(x) = \exp(h)$ with small norm, we can also assume that $\| h \|_{\hs} \leq 2 N^{2} \phi_{n}^{-1} \epsilon_{n}$. Since $\d f_{0}$ has full rank, we can write $h = \d f_{0}(y)$ for some $y$.  We calculate:  
\be \label{ExpClosenessToBoundFuzz}
\| x- y \|_{\infty} \leq \| x - y \| &\leq 2 \phi_{n}^{-1} \| \d f_{0}(x) - \d f_{0}(y) \|_{\hs} \\
&= 2 \phi_{n}^{-1} \| h - \d f_{0}(x) \|_{\hs} \\
&\leq 4 \phi_{n}^{-1} \| \exp(h) - \exp(\d f_{0}(x)) \|_{\hs} \\
&= 4 \phi_{n}^{-1}  \| f(x) - \exp(\d f_{0}(x)) \|_{\hs} \\
&\leq 32 N^{2} \phi_{n}^{-1}  \epsilon_{n}^{2}.
\ee 
The first line follows from inequality \eqref{IneqTangentMapNotSmallPartTwo}. The third line follows for $n$ sufficiently large from a Taylor expansion and the fact that $\epsilon_{n} \ll n^{-4}$. The last line follows from inequality \eqref{IneqCoupCompToExp}. This proves that if  $x \in \partial[-c, c]^{N}$ and $f(x) = \exp(\d f_{0}(y))$, then  $y \notin [-u_{1}, u_{1}]^{N}$. 

This implies that $f(\partial[-c,c]^{N}) \subset \mathcal{H}_{f}(u_{1})^{c}$. Since $f$ is a diffeomorphism  by Lemma \ref{LemmaDiff} and a map between manifolds of the same dimension $N$, this implies $f(\partial[-c,c]^{N}) = \partial f([-c,c]^{N}) \subset \mathcal{H}_{f}(u_{1})^{c}$. Using the fact that both $f$ and the exponential map are diffeomorphisms, the condition $\partial f([-c,c]^{N}) \subset \mathcal{H}_{f}(u_{1})^{c}$ together with the fact that $f(0) \subset f([-c,c]^{N}) \cap \mathcal{H}_{f}(u_{1}) \neq \emptyset$, implies the containment condition $\mathcal{H}_{f}(u_{1}) \subset  f([-c,c]^{N})$. This is exactly the left-hand side of the first containment condition \eqref{IneqContainments}.

To prove the right-hand side \eqref{IneqContainments}, essentially the same calculation shows that for any $p = \exp(\d f_{0}(y)) \in f([-c,c]^{N})$, we have $y \in [-u_{2}, u_{2}]^{N}$. This immediately implies the right-hand side of the first containment condition \eqref{IneqContainments}. 

The proof of the second containment condition \eqref{IneqContainments} is identical, so this completes the proof of the proposition.
\end{proof}

Since the exponential map is surjective and $\d f_{0}$ has full rank, there exists some $h$ so that $g(0) = \L_{f(0)} \exp( \d\L_{f(0)^{-1}} \d f_{0}(h))$. By the fact that the exponential map takes geodesic paths and the assumption that $\epsilon_{n} = o(n^{-5})$, followed by inequality \eqref{IneqCoupCond1}, we have
\be \label{TrivialNormBoundCoup}
\|h \|_{\hs} \leq \|f(0) - g(0) \|_{\hs} \leq \omega_{n} n^{5 +C}.
\ee 

We claim: 

\begin{prop}  \label{LemmaBoxTwist}
Fix $c > 0$,  set $u_{1}$, $u_{2}$ as in equation \eqref{IneqOrigBoxCond}, and set 
\be 
v_{1} &= u_{1} - 64 N^{2} \phi_{n}^{-1}  \epsilon_{n}^{2} \\
v_{2} &= u_{2} + 64 N^{2} \phi_{n}^{-1}  \epsilon_{n}^{2}. 
\ee 
Then for all $n$ sufficiently large,
\be \label{EqContainmentTwoConc}
\mathcal{H}_{f}(v_{1}) \subset \mathcal{H}_{g}(u_{1}) \subset  \mathcal{H}_{g}(u_{2}) \subset \mathcal{H}_{f}(v_{2})
\ee 
uniformly in $c = c_{n}$ satisfying $\frac{1}{2} \epsilon_{n} \leq c \leq 2 \epsilon_{n}$.
\end{prop}

\begin{proof}
Assume that the left-most containment in \eqref{EqContainmentTwoConc} is not true.  Then there exists $p \in \mathcal{H}_{f}(v_{1}) \backslash \mathcal{H}_{g}(u_{1})$. Write $p = f(0)\, \exp(\d \L_{f(0)^{-1}} \circ \d f_{0}(x))$ for some $x \in [-v_{1},v_{1}]^{N}$. By the local surjectivity of the exponential map, we can write $p = 
g(0) \,\exp(\d \L_{g(0)^{-1}} \circ \d g_{0}(y))$ for some $y$. Since $p \notin \mathcal{H}_{g}(u_{1})$, we have $y \notin [-u_{1}, u_{1}]^{N}$. However, by essentially the same calculation as in \eqref{ExpClosenessToBoundFuzz} combined with inequality \eqref{TrivialNormBoundCoup},
\be 
\| x - y \|_{\infty} \leq \| x - y \| \leq 32 N^{2} \phi_{n}^{-1}  \epsilon_{n}^{2} +  \omega_{n} n^{5 +C} \leq 64 N^{2} \phi_{n}^{-1}  \epsilon_{n}^{2}.
\ee  
Thus,
\be 
v_{1} \geq \| x \|_{\infty} \geq \|y \|_{\infty} - 64 N^{2} \phi_{n}^{-1}  \epsilon_{n}^{2} > u_{1} -  64 N^{2} \phi_{n}^{-1}  \epsilon_{n}^{2} = v_{1}.
\ee 
This is a contradiction, and so no such $p$ exists. This completes the proof of the first containment relationship in \eqref{EqContainmentTwoConc}. The second containment is trivial, and the third is proved in essentially the same way as the first. This completes the proof of the proposition.
\end{proof}

Combining Propositions \ref{PropOriginalBoxes} and \ref{LemmaBoxTwist}, we have
\be \label{IneqFinalContainmentConditionReally}
f([-\epsilon_{n},\epsilon_{n}]^{N}) \supset \mathcal{H}_{g}(c_{1}) \supset g([-c_{2},c_{2}]^{N}),
\ee 
where $c_{1} = \epsilon_{n} - 256 N^{2} \phi_{n}^{-1}  \epsilon_{n}^{2}$ and $c_{2} = c_{1} - 256 N^{2} \phi_{n}^{-1}  \epsilon_{n}^{2}$. 

Let $\rho_{f}$ and $\rho_{g}$ denote the densities of $\mathcal{L}(f(U_{1},\ldots, U_{N}))$ and $\rho_{g}$ of $\mathcal{L}(g(V_{1},\ldots, V_{N}))$. By Lemma \ref{LemmaJacobianBound},
\be 
| \frac{\rho_{g}(x)}{\rho_{f}(y)} - 1 | \leq N^{\frac{N}{2}} \left( \frac{ \| df_{x} - dg_{y} \|_{\mathrm{Op}} }{\sigma_{1}(df_{x}) } \right)^{N}.
\ee 

By Lemma \ref{LemmaTangentClose} and assumption \eqref{IneqTangentMapNotSmallPartTwo}, $ \frac{ \| df_{x} - dg_{y} \|_{\mathrm{Op}} }{\sigma_{1}(df_{x}) } \leq 32 \phi_{n}^{-1} N^{2} \epsilon_{n}$, and so 
\be \label{JacApprApplication}
| \frac{\rho_{g}(x)}{\rho_{f}(y)} - 1 | \leq (32 N^{2.5} \phi_{n}^{-1} \epsilon_{n})^{N} \ll N^{-N}.
\ee 

Combining inequality \eqref{JacApprApplication} with the containment condition \eqref{IneqFinalContainmentConditionReally},
\be 
\| \mathcal{L}(f(U_{1},\ldots, U_{N})) - \mathcal{L}(g(V_{1},\ldots, V_{N})) \|_{\TV} &\leq 1 - \frac{\mu(f([-\epsilon_{n},\epsilon_{n}]^{N}) \cap g([-\epsilon_{n},\epsilon_{n}]^{N}))}{\mu(g([-\epsilon_{n},\epsilon_{n}]^{N}))}(1 + N^{-N}) \\
&\leq 1- \frac{\mu(g([-c_{2},c_{2}]^{N}))}{\mu(g([-\epsilon_{n},\epsilon_{n}]^{N}))} (1 + N^{-N}). \label{MinCondFirst} \\
&= 1 - \frac{(2 c_{2})^{N}}{ (2 \epsilon_{N})^{N}}(1 + N^{-N})\\
&\leq 1 - \frac{(2 \epsilon_{n} -  1024 N^{2} \phi_{n}^{-1}  \epsilon_{n}^{2})^{N}}{(2 \epsilon_{n})^{N}}(1 + N^{-N}) \\
&\leq 1 - 513 N^{2} \phi_{n}^{-1}  \epsilon_{n}.
\ee  

This completes the proof of the lemma.
\end{proof}

\section{Relating Mixing Times to Singular Values} \label{SecMainKacThmStrengthening}

Let $f_{\mathcal{A},\epsilon_{n}}$ be as in Equation \ref{EqDefMainMap}, and let $D$ be the Jacobian of $f_{\mathcal{A},\epsilon_{n}}$ at $(0,0,\ldots,0)$. In this section, we apply our current bound to relate the mixing time of Kac's walk to the smallest singular value $\sigma_{1}(D)$ of the matrix $D$. Define the scaling sequence $\{\phi_{n}\}_{n \geq 1}$ by

\be  \label{DefScalingSingVal}
\phi_{n} = \min((2n)^{-30}, \inf \{ r > 0 \, : \, \P[\sigma_{1}(D) \leq r] < \frac{1}{\sqrt{n}}).
\ee

We give most of the proof of Theorem \ref{MainKacThm}, deferring the proof of a  bound on $\phi_{n}$ to Section \ref{SecRandomMatrix}. We begin with the lower bound:

\begin{thm} [Lower Bound] \label{LemmaLowerBound}
For $T < N$ and any $X_{0} = x \in \so$,
\be 
 \| \mathcal{L}(X_{T}) - \mu \|_{\TV} = 1.
\ee 
\end{thm}

\begin{proof}
The proof is essentially a matter of counting dimensions. 

Let $\{X_{t} \}_{t \geq 0}$ be a copy of Kac's walk, and let $\{ i_{t}, \theta_{t} \}_{t \geq 0}$ be the associated update variables as described in Equation \eqref{KacRep}. For $t \in \mathbb{N}$, define $f_{t}:[0,2\pi)^{t} \mapsto \so$ by 
\be 
f_{t}(x_{0},\ldots, x_{t-1}) = \prod_{s=t-1}^{0} \R(i_{s},x_{s}) X_{0}
\ee 
and define the set $A_{t} = A_{t}(\{i_{t} \}_{t \geq 0}) \equiv f_{t}([0,2\pi)^{t})$. We note that $X_{t} \in A_{t}$. From this definition, $\text{Rank}((\d f_{t})_{p}) \leq t < N$ for all $p$. By Sard's theorem (see pp. 205 of \cite{GuillemanPollack74DiffTop}), this implies 
\be \label{IneqCouponZeroMeasure} 
\pi(A_{t}) =  0.
\ee 
Next, define 
\be 
\mathrm{Max}_{T} = \cup_{I \in \{1,2,\ldots,N \}^{T}} A_{T}(I).
\ee 
Since this is a union of only ${n \choose 2}^{T} < \infty$ elements, Equation\eqref{IneqCouponZeroMeasure}  implies
\be 
\pi(\mathrm{Max}_{T}) = 0.
\ee 
Thus,
\be 
\| \mathcal{L}(X_{T}) - \pi \|_{\TV} \geq | \P[X_{t} \in\mathrm{Max}_{T}] - \pi(\mathrm{Max}_{T}) | = 1,
\ee 
finishing the proof.
\end{proof}

\begin{remark} [Dimension Counting and Curved Spaces]
A natural approach for obtaining a lower bound is to count the dimension of the tangent map associated with the function $f_{t}$ at $0$ rather than bounding the dimension of the image of $f_{t}$ itself. This approach suggests that the chain will  not have mixed until the first time $T$ that the span of $\{a_{i_{1}}, a_{i_{2}},\ldots,a_{i_{T}}\}$ has dimension $N$ with high probability; by the usual coupon collector argument, this requires $T \gtrsim n^{2} \log(n)$. 

While this approach works for Gibbs samplers on convex sets (for which the dimension of the tangent map of $f_{t}$ at 0 is an upper bound on the dimension of the image of $f_{t}$), and it works for Kac's walk on the sphere for different reasons (see \cite{pillai2015kac}), it does \textit{not} work for Kac's walk on the sphere. In particular, it is possible for $f_{t}$ to have full dimension $N$, despite $\mathrm{Dimension}(\mathrm{span}(\{a_{i_{1}}, a_{i_{2}},\ldots,a_{i_{t}}\})) < N$. See the famous `Euler angle' decomposition of $\mathrm{SO}(3)$ for an illustration of this fact \cite{ClassicalMechanics01}.
\end{remark}

The upper bound in Theorem \ref{MainKacThm} will be a corollary of the following result: 

\begin{thm} [Intermediate Bound on the Mixing Time of Kac's Walk on $\so$] \label{MainKacThmStrengthening}
Fix $0 < \mathcal{Q} < \infty$ and consider the coupling given in Section \ref{DefCoupLongDesc}. Let $\phi_{n}$ be as defined in Equation \eqref{DefScalingSingVal}, and let $R = 8 \mathcal{Q} n^{4} \log(n)$. Assume that $\phi_{n}=o(n^{n^{2.2}})$. Then
\be \label{IneqMainThmUpper}
\limsup_{n \rightarrow \infty} \sup_{X_{0} = x \in \so} \| \mathcal{L}(X_{t}) - \mu \|_{\TV} = 0
\ee 
for any sequence $t = t(n) > R + 5 n^{2} \log(n) + 900 n^{2} \log(\phi_{n})$.
\end{thm}

\begin{proof}
Let $\phi_{n}$ be as in Equation \eqref{DefScalingSingVal}, let $\epsilon_{n} = \phi_{n}^{30}$ and let $\omega_{n} = \epsilon_{n}^{30}$. Fix $R < S < N^{4.9}$. We couple two copies of Kac's walk $\{X_{t} \}_{t = 0}^{S}$, $\{Y_{t} \}_{t = 0}^{S}$ started at positions satisfying $\| X_{0} - Y_{0} \|_{\hs} \leq \omega_{n} n^{-5}$ according to the non-Markovian coupling given in Section \ref{DefCoupLongDesc}.

Let $\mathcal{E}_{1}$ be the event that $f_{\mathcal{A}, \epsilon_{n}}$ and $f_{\mathcal{B}, \epsilon_{n}}$ are well-defined at time $S$ and condition \eqref{IneqCoupCond1} is satisfied at time $S$ with constant $C = \frac{\log(\omega_{n})}{\log(\epsilon_{n})} - 10 = 20$, and let $\mathcal{E}_{2}$ be the event that $f_{\mathcal{A}, \epsilon_{n}}$ and $f_{\mathcal{B}, \epsilon_{n}}$ are well-defined at time $S$ and condition \eqref{IneqTangentMapNotSmallPartTwo} is satisfied at time $S$. The conclusion of Lemma \ref{LemmaCoupConAbs}, together with Remark \ref{RemarkCoupGoal}, implies that 
\be \label{IneqPuttingItTogether}
\| X_{S} - Y_{S} \|_{\TV} \leq (1 - \P[\mathcal{E}_{1}] - \P[\mathcal{E}_{2}]) + o(1).
\ee 

By Lemma \ref{LemmaDimensionTime},  $\P[s_{N} < R] = 1 - o(1)$, and so the functions $f_{\mathcal{A}, \epsilon_{n}}$ and $f_{\mathcal{B}, \epsilon_{n}}$ are well-defined for all times $S > R$ with probability $1 - o(1)$. By Lemma \ref{LemmaPathCloseness}, the condition \eqref{IneqCoupCond1} of Lemma \ref{LemmaCoupConAbs} is satisfied with probability $1 - o(1)$. Thus, $\P[\mathcal{E}_{1}] = 1 - o(1)$. By the definition of $\phi_{n}$, we have that $\P[\mathcal{E}_{2}] = 1 - o(1)$. Combining these bounds with Inequality \eqref{IneqPuttingItTogether}, we conclude that 
\be 
\| X_{S} - Y_{S} \|_{\TV} = o(1).
\ee

By Lemma \ref{LemmaOlivBound}, this implies for $\{X_{t} \}_{t \geq 0}$ a copy of Kac's walk starting at any point $X_0 = x \in \so$ and $T > R + 5 n^{2} \log(n) + 900 n^{2} \log(\phi_{n})$,
\be 
\sup_{X_{0} = x \in \so} \| \mathcal{L}(X_{T}) - \mu \|_{\TV} = o(1).
\ee 
This completes the proof.
\end{proof}

The following, similar, result allows us to calculate much sharper numerical bounds on the mixing time of Kac's walk:

\begin{thm} [Alternative Intermediate Bound on the Mixing Time of Kac's Walk on $\so$] \label{MainKacThmOptimistic}

Consider the coupling given in Section \ref{DefCoupLongDesc}, with the reference to Definition \ref{DefSubsetChoice} replaced by a reference to Definition \ref{DefSubsetChoiceGreedy}. Let $\phi_{n}$ be as defined in Equation \eqref{DefScalingSingVal}. Then
\be \label{IneqMainThmUpper}
\limsup_{n \rightarrow \infty} \sup_{X_{0} = x \in \so} \| \mathcal{L}(X_{t}) - \mu \|_{\TV} = 0
\ee 
for any sequence $t = t(n)$ satisfying $\lim_{n \rightarrow \infty} \frac{t(n)}{n^{2} \log(n \phi_n)} = \infty$.
\end{thm}

\begin{proof}
The proof is almost identical to that of Theorem \ref{MainKacThmStrengthening}. The main difference is that all references to Lemma \ref{LemmaDimensionTime} should be replaced by references to Lemma \ref{LemmaDimensionTimeGreedy}. \\
\end{proof}

\begin{remark}
In either case, we note that this implies the mixing time of Kac's walk satisfies $\tmix = O(n^{2} \log(\phi_{n}))$.
\end{remark}

\begin{remark} [Estimating Mixing Times via Simulation] \label{SecComp}
We do not analyze the bound from Theorem \ref{MainKacThmOptimistic} in this paper. We include it because it can be used to obtain rigorous upper bounds on the mixing time of Kac's walk by simulation. In particular, the random matrix $D$ that appears in Equation \eqref{DefScalingSingVal} can easily be simulated on a computer. Thus, the quantiles of the distribution of $\sigma_{1}(D)$ can be estimated by simulation, which allows us to calculate upper bounds on the constant $\phi_{n}$ with high confidence. 

We point out that the existence of a method to rigorously bound the mixing time of a Gibbs sampler is not obvious. Indeed, the authors are not aware of \textit{any} finite computation that allows one to rigorously bound the mixing time of a generic Gibbs sampler on a continuous non-convex state space. We also mention that, just as the coupling in Section \ref{DefCoupLongDesc} makes sense for any Gibbs sampler, this computational approach to bounding the mixing time can be extended (with some effort) to many other Gibbs samplers.
\end{remark}

\section{ Singular Values of Random Matrices} \label{SecRandomMatrix}

The last ingredient in the proof of our upper bound on the mixing time of Kac's walk is a \textit{lower} bound on the smallest singular value $\sigma_{1}(D)$ of the matrix $D$ defined in Equation \eqref{EqJacMatCompRedux}. In this section, we obtain the required bound.

We begin by giving a generic bound on the smallest singular value of a random matrix whose entries have continuous but strongly dependant entries (see Section \ref{SecGenericRandomMat}), then apply this bound to a simple random matrix $D_{\infty}$ for which some exact calculations are possible (see Section \ref{SecToyApplicationRandMat}), and finally compare the smallest singular value of $D_{\infty}$ to that of the matrix $D$ defined in Equation \eqref{EqJacMatCompRedux} (see Section \ref{SubsecApplicationsRandomMatKac}). This argument gives us a lower bound on the constant $\phi_{n}$ that is defined in Equation \eqref{DefScalingSingVal}, and which has a critical role in the conclusion of Theorem \ref{MainKacThmStrengthening}. 

We believe that the bounds in this section may be of independent interest. The notation used in this section is also essentially independent of the notation of the remainder of the paper, except where explicitly noted. Our main abstract results, given in Lemmas \ref{LemmaGenMatrixDetBound} and 
\ref{LemmaSmallBallBound}, are qualitatively similar to the main bounds in \cite{friedland2013simple}. While our assumptions are similar to those in \cite{friedland2013simple}, the key difference is that our results apply to matrices with a great deal more dependence and which may be symmetric. Finally, the paper \cite{farrell2015smoothed} gives related but much stronger conclusions than our paper or \cite{friedland2013simple}, but under much stronger independence assumptions.

\subsection{Bounds on Determinant of Random Matrices} \label{SecGenericRandomMat}

Let $M$ be an $n \times n$ symmetric random matrix with associated $\sigma$-algebra $(\Omega, \Sigma)$. Let $\mathcal{F}_0$ denote the $\sigma$-algebra generated by the entries of $M$. For $1 \leq i  \leq n-1$, let $\mathcal{F}_{i}$ be a $\sigma$-algebra under which $M[k,\ell]$ is $\mathcal{F}_{i}$-measurable for all $(k,\ell)$ satisfying either $i \leq k$ and $ i +2 \leq \ell$, or $k = \ell \in \{i,i+1\}$.
Finally, let $\zeta_{i} \sim \mathcal{L}(M[i,i+1] | \mathcal{F}_{i})$. We make the following assumptions for the matrix $M$:

\begin{assumption}\label{assum:M}

\begin{enumerate}
\item The random variable $\zeta_i$ satisfies the anti-concentration bound
\be \label{IneqRandomMatDensityAssumption}
\P\Big[ \sup_{x \in \mathbb{R}} \sup_{\beta \in \mathbb{R}, (\alpha,\epsilon) \in \mathcal{R}(C)}\P[ | \alpha(\zeta_{i})^{2} + \beta \zeta_i - x | < \epsilon| \mathcal{F}_i ] < \frac{4 C \sqrt{\epsilon}}{\sqrt{\alpha}} + n^{-2} \Big] > 1- n^{-2}
\ee  
for some fixed $1 \leq C < \infty$, where 
\be 
\mathcal{R}(C) = \{ (\alpha, \epsilon); \alpha > 0, \epsilon \geq (4C^2n^4)^{-1} \alpha \}.
\ee
\item We have 
\be [IneqBaseCaseDensity]
\P[| M[n,n] | < (4Cn)^{-4}] \leq \frac{1}{n^{2}}, \qquad n \text{ odd.} \\
\P[| M[n-1,n-1]M[n,n] - M[n-1,n]^{2}| < (4Cn)^{-4}] \leq \frac{1}{n^{2}}, \qquad n \text{ even.} \\
\ee 
\end{enumerate}
\end{assumption}

\begin{remarks}
The assumption given by Inequality \eqref{IneqRandomMatDensityAssumption} is often easy to verify in practice. For example, it holds if the conditional density $\rho_{i}$ of $\zeta_{i}$ is bounded by the constant $1 \leq C < \infty$ with high probability (see Lemma \ref{LemmaSmallBallBound}). The second part of Assumption \ref{assum:M} is often straightforward to check by hand.
\end{remarks}

Let $|M|$ denote the determinant of the matrix $M$.
\begin{lemma} \label{LemmaGenMatrixDetBound}
For a matrix $M$ satisfying the hypotheses of Assumption \ref{assum:M}, 
\be 
\P[|M| < (4Cn)^{-4(n+1)} ] \leq \frac{3}{n}
\ee 
for all $n \in \mathbb{N}$.
\end{lemma}

\begin{proof}
For an $n \times n$ matrix $A$ and indices $1 \leq i,j \leq n$, denote by $A_{(i,j)}$ the matrix obtained by removing the $i$'th row and $j$'th column from $A$, and $A_{(i,j), (k,\ell)}$ the matrix obtained by removing the $i$'th and $k$'th rows and the $j$'th and $\ell$'th columns.

Let $m = \frac{n-1}{2}$ when $n$ is odd, and let $m = \frac{n}{2}-1$ when $n$ is even. Let $M^{(1)} = M$. We inductively define $\{M^{(k)}\}_{k=1}^{m}$ by setting
\be 
M^{(k+1)} = M^{(k)}_{(1,2), (2,1)}.
\ee

Observe that $M^{(k)}$ is just the $(n-2k+2)$ by $(n-2k+2)$ lower-right-hand submatrix of $M = M^{(1)}$. We define the events:

\be 
U_{k} &= \{ | M^{(k)} |^{2} > (4 C n)^{-4(m-k+2)} \} \\
V_{k} &=\Big \{ \sup_{x \in \mathbb{R}} \sup_{\beta \in \mathbb{R}, (\alpha,\epsilon) \in \mathcal{R}(C)}\P[ | \alpha(\zeta_{2k-1})^{2} + \beta \zeta_{2k-1} - x | < \epsilon| \mathcal{F}_{2k-1} ] < \frac{4 C \sqrt{\epsilon}}{\sqrt{\alpha}} + n^{-2} \Big \}.
\ee 
By definition $U_k \in \mathcal{F}_{2k-2}$ and $V_k \in \mathcal{F}_{2k-1}$.
We now expand the determinant of $M^{(k)}$:
\be 
|M^{(k)}| &= - M^{(k)}[1,2] \, | M^{(k)}_{(1,2)} | + C_{1} \\
&= - M^{(k)}[1,2] M^{(k)}[2,1] \, | M^{(k)}_{(1,2),(2,1)}| + M^{(k)}[1,2] C_{2} + C_{1} \\
&= -M^{(k)}[1,2]^{2} \,  |M^{(k+1)}| + M^{(k)}[1,2] C_{2} + C_{1}, \label{eqn:Mk}
\ee 
where 
\be 
C_{1} &= \sum_{1 \leq j \leq n, \, j \neq 2} (-1)^{j+1} M^{(k)}[1,j] \, |M^{(k)}_{(1,j)}| \\
C_{2} &= \sum_{3 \leq j \leq n} (-1)^{j} M^{(k)}[2,j] \, |M^{(k)}_{(1,2),(2,j)}|.
\ee 
By assumption, $C_1, C_2 \in \mathcal{F}_{2k-1}$. Thus from Equation \ref{eqn:Mk} and choosing $\epsilon^2 = {1\over 16} C^{-4} n^{-8}|M^{(k+1)}|^2$ and $\alpha =  |M^{(k+1)}|$ in Equation \ref{IneqRandomMatDensityAssumption}, we obtain 
\be \label{eqn:keyMk}
\P\Big(|M^{(k)}|^2 < {1 \over 16} C^{-4} n^{-8} |M^{(k+1)}|^{2}| V_k\Big) > 1-{2 \over n^2}.
\ee
Using Equation \eqref{eqn:keyMk} we deduce that
\be 
\P[U_{k}^{c} \cap U_{k+1} \cap V_{k}] &= \mathbb{E}\Big[1_{U^c_k} 1_{U_{k+1}} 1_{V_k} \Big] \\
&\leq \mathbb{E}\Big[1_{U^c_k} | V_{k}, U_{k+1}\Big] \\
&\leq \frac{4C (4 C n)^{-2(m-k+2)}}{(4 C n)^{-2(m-k+1)}} + n^{-2} \leq 2 n^{-2}.
\label{IneqRandMatMainCalc}
\ee 
Using this inequality repeatedly, and defining $V = \cup_{i} V_{i}$, we have 
\be 
\P[U_{1}^{c}] &\leq \P[U_{1}^{c} \cap V] + \P[V^{c}] \\
&\leq \P[U_{1}^{c} \cap V] + \frac{m}{n^{2}} \\
&= \P[U_{1}^{c} \cap U_{2} \cap V] + \P[U_{1}^{c} \cap U_{2}^{c} \cap V] + \P[V^{c}] \\
&= \ldots \\
&= \sum_{j=2}^{m} \P[U_{1}^{c} \cap \ldots \cap U_{j-1}^{c} \cap U_{j} \cap V] + \P[U^{c}_m] + \P[V^{c}]\\
&\leq 2 \frac{m-1}{n^{2}} + \frac{1}{n^{2}} + \frac{m}{n^{2}} \leq \frac{3}{n},
\ee 
where the first few lines are repeated use of exact equalities and union bounds, and the three terms in the last line use, respectively, Inequality \eqref{IneqRandMatMainCalc}, Inequality \eqref{IneqBaseCaseDensity} and the assumption in Inequality \eqref{IneqRandomMatDensityAssumption} that $\P[V^{c}] \leq n^{-2}$. This completes the proof.
\end{proof}

We give a simple sufficient condition for the assumption given by Inequality \eqref{IneqRandomMatDensityAssumption} to hold:

\begin{lemma} [Non-Concentration] \label{LemmaSmallBallBound}
Define $f \, : \, \mathbb{R} \mapsto \mathbb{R}$ by
\be 
f(x) = \alpha x^{2} + \beta x + \gamma
\ee 
and let $X$ be a random variable with density $\rho$ satisfying $\sup_{x} \rho(x) < C < \infty$. Then for any $x \in \mathbb{R}$ and $\epsilon > 0$,
\be 
\P[|f(X) - x| < \epsilon] \leq \frac{4 C \sqrt{\epsilon}}{\sqrt{|\alpha|}}.
\ee 
\end{lemma}

\begin{proof}

Assume without loss of generality that $\alpha > 0$, and let $\eta$ be the density of $f(X)$. Fix $r \in \mathbb{R}$ and define the quantity
\be
\Gamma = {r \over \alpha} + \frac{\beta^{2}}{4 \alpha^{2}} - \frac{\gamma}{\alpha}.
\ee
We have
\be 
\P[f(X) \leq r] &= \P[(X - \frac{\beta}{2 \alpha})^{2} \leq \Gamma] \\
&=\P[X \leq \frac{\beta}{2 \alpha} + \sqrt{\Gamma}] - \P[X \leq \frac{\beta}{2 \alpha} - \sqrt{\Gamma}].
\ee 
Thus
\be 
\eta(r) &= \frac{d}{dr} \P[f(X) \leq r] \\
&= \rho(\frac{\beta}{2 \alpha} + \sqrt{\Gamma})  \frac{1}{2 \alpha \sqrt{\Gamma}} + \rho(\frac{\beta}{2 \alpha} - \sqrt{\Gamma})  \frac{1}{2 \alpha \sqrt{\Gamma}} \\
&\leq \frac{C}{\alpha} \frac{1}{\sqrt{\Gamma}}.
\ee 
Thus, we have 
\be 
\P[ |f(X) - x| \leq \epsilon] &= \int_{x-\epsilon}^{x + \epsilon} \eta(r) dr \\
&\leq \frac{C}{\alpha} \int_{x-\epsilon}^{x+\epsilon} \frac{1}{\sqrt{\Gamma}} \\
&\leq \frac{C}{\alpha} \int_{-\epsilon}^{\epsilon} \frac{\sqrt{\alpha}}{\sqrt{|r|}} dr \\
&= \frac{4 C \sqrt{\epsilon}}{\sqrt{\alpha}}.
\ee 
This completes the proof of the bound.
\end{proof}

We next show that the assumption given by Inequality \eqref{IneqRandomMatDensityAssumption} remains true under small perturbations. For two probability measures $\nu_1, \nu_2$ on $\mathbb{R}^d$, define the \textit{Wasserstein distance}
\be
W_{2}(\nu_1,\nu_2)^{2} = \inf_{(X,Y) \in \mathcal{C}, X\sim \nu_1, 
Y \sim \nu_2} \E[ \|X-Y\|_{2}^{2}],
\ee
where $\mathcal{C}$ is the set of all couplings on $\mathbb{R}^d\times \mathbb{R}^d$ with marginal distributions $\nu_1$ and $\nu_2$.

Let $M'$ be another $n \times n$ symmetric random matrix, and let $\mathcal{F}_{i}'$, $\zeta_{i}$', etc be defined analogously to $\mathcal{F}_{i}$, $\zeta_{i}$. For fixed $1 \leq i \leq n$, let $\mathcal{G}_{i} = \mathcal{G}_{i}(\{m_{k,\ell}\}) $ (respectively $\mathcal{G}_{i}'(\{m_{k,\ell}\})$) be the event that $M[k,\ell] = m_{k,\ell}$ (respectively $M'[k,\ell] = m_{k,\ell}$) for all $k,\ell \geq i+2$.
\begin{lemma} \label{LemmaRandomMatComparison}
Let the density $\rho_{i}$ of $\zeta_{i}$ satisfies 
\be \label{IneqCompInitialGoodMatDensity}
\P[\sup_{x} \rho_{i}(x) > C] \leq \frac{1}{2 n^{5}}
\ee for some fixed $1 \leq C < \infty$.
Let $M'$ satisfy \eqref{IneqBaseCaseDensity} of Assumption \ref{assum:M}.
 Assume that 
\be \label{RandomMatCompWassClosenessCond}
W_{2}( \mathcal{L}(M[i,(i+1):n] \, | \, \mathcal{G}_{i}), \mathcal{L}(M'[i,(i+1):n] \, | \, \mathcal{G}_{i}') )^{2} \leq \delta < \frac{1}{8  n^{5}} (4Cn)^{-4}.
\ee 
Then there exists a universal constant $N_{0}$ so that the determinant $|M'|$ of $M'$ satisfies: 
\be 
\P[|M'| < (4Cn^{3})^{-4(n+1)} ] \leq \frac{3}{n}
\ee 
for all $n > N_{0}$.
\end{lemma}

\begin{proof}
We begin with a generic bound. Let $X,Y \in \mathbb{R}$ be two random variables, and let $f \, : \, \mathbb{R} \mapsto \mathbb{R}$ by $f(x) = \alpha x^{2} + \beta x$ for some $\alpha > 0$. Fix $\epsilon > 0$ and $x \in \mathbb{R}$. Then 
\be 
\sup_{x \in \mathbb{R}} \P[ |f(X) - x| < \epsilon] &= \sup_{x \in \mathbb{R}}  \P[ |\alpha X^{2} + \beta X - x| \leq \epsilon] \\
&\leq \sup_{x \in \mathbb{R}} \P\Big[X \in  \Big(x - \sqrt{\frac{\epsilon}{\alpha}}, x + \sqrt{\frac{\epsilon}{\alpha}}\Big) \Big] \\
&\leq \sup_{x \in \mathbb{R}} \P[Y \in  [x - \sqrt{\frac{\epsilon}{4 \alpha}}, x + \sqrt{\frac{\epsilon}{4 \alpha}}] ] + \P[|X-Y| > \sqrt{\frac{\epsilon}{4 \alpha}}] \\
&\leq \sup_{x \in \mathbb{R}} \P\Big[Y \in  \Big(x - \sqrt{\frac{\epsilon}{4 \alpha}}, x + \sqrt{\frac{\epsilon}{4 \alpha}}\Big) \Big] + W_{2}(\mathcal{L}(X),\mathcal{L}(Y))^{2} \frac{4 \alpha}{\epsilon}. \label{eqn:fquadwass}
\ee 
By Equation \eqref{IneqCompInitialGoodMatDensity} and Lemma \ref{LemmaSmallBallBound}, Assumption \ref{assum:M} is satisfied the by the matrix $M$.
For $\epsilon$, $f$ as above, Equation \eqref{eqn:fquadwass} then implies
\be 
\sup_{x \in \mathbb{R}}  \P[ f(M'[i,i+1]) & \in [x-\epsilon,x+\epsilon] \, | \, \mathcal{G}_{i}']  \\
&\leq \sup_{x \in \mathbb{R}} \P[ f(M[i,i+1]) \in [x-\epsilon, x + \epsilon] \, | \, \mathcal{G}_{i}]  \\
&\hspace{2cm}+ \frac{4 \alpha}{\epsilon} W_{2}\Big( \mathcal{L}(M[i,(i+1):n] \, | \, \mathcal{G}_{i}), \mathcal{L}(M'[i,(i+1):n] \, | \, \mathcal{G}_{i}') \Big)^{2} \\
&\leq \frac{1}{2n^{5}} + \frac{4C \sqrt{\epsilon}}{\alpha} + \frac{4 \alpha}{\epsilon} W_{2}\Big( \mathcal{L}(M[i,(i+1):n] \, | \, \mathcal{G}_{i}), \mathcal{L}(M'[i,(i+1):n] \, | \, \mathcal{G}_{i}')\Big)^{2} \\
&\leq  \frac{1}{2n^{5}} + \frac{4C \sqrt{\epsilon}}{\alpha} + \frac{4 \alpha \delta }{\epsilon}, \\
\ee 
where the first term in the second-last line comes from Inequality \eqref{IneqCompInitialGoodMatDensity}.
 and the second term comes from an application of Lemma \ref{LemmaSmallBallBound}, and where the last line comes from Inequality \eqref{RandomMatCompWassClosenessCond}.
This bound implies
\be 
\E[ \sup_{x \in \mathbb{R}} \P[ f(M'[i,i+1]) \in [x-\epsilon,x+\epsilon] \, | \, \mathcal{F}_{i}'] \, | \, \mathcal{G}_{i}'] \leq \frac{1}{2n^{5}} + \frac{4C \sqrt{\epsilon}}{\alpha} + \frac{4 \alpha \delta }{\epsilon} \leq \frac{1}{n^{5}} + \frac{4C \sqrt{\epsilon}}{\alpha} .
\ee 
By Markov's inequality, this implies
\be 
 \P[ \sup_{x \in \mathbb{R}} \P[ f(M'[i,i+1]) \in [x-\epsilon,x+\epsilon] \, | \, \mathcal{F}_{i}'] > n^{2}(n^{-5} + \frac{4C \sqrt{\epsilon}}{\alpha})] \leq n^{-2}.
\ee 
Thus $M'$ satisfies Assumption \ref{assum:M} with constant $\tilde{C} = n^{2} C$. Applying Lemma \ref{LemmaGenMatrixDetBound} now completes the proof.
\end{proof}

\subsection{Bounding the Smallest Singular Value of $D_{\infty}$}\label{SecToyApplicationRandMat}

In this section, we define a specific random matrix $D_{\infty}$, and show that it satisfies the requirements of Lemma \ref{LemmaGenMatrixDetBound}. 
Define the $n$-sphere 
\be
\mathrm{S}^{(n-1)} = \{x \in \mathbb{R}^{n} \, : \, \sum_{i=1}^{n} x[i]^{2} = 1\}.
\ee
Denote the Euclidean inner product by $\langle \cdot , \cdot \rangle$. We begin with the technical lemma:

\begin{lemma} [Conditional Densities on Spheres] \label{LemmaCondDensSphereBound}
Fix $0 \leq k \leq n-2$. Let $X, v_{1},\ldots,v_{n-1} \sim \mathrm{Unif}(\mathrm{S}^{(n-1)})$ be i.i.d.  For $0 \leq k \leq n-2$, let $\mathcal{G}_{k}$ be the $\sigma$-algebra generated by $ \langle X, v_{1} \rangle, \ldots, \langle X, v_{k} \rangle$ and $v_{1},\ldots,v_{n-1}$.  Let
\be 
Z \sim \mathcal{L}( \langle X, v_{n-1} \rangle \, | \,\mathcal{G}_{k}),
\ee 
and let $\rho$ be the density of $Z$. We have 
\be 
\P[\sup_{z} \rho(z) > n^{20}] \leq n^{-2}
\ee 
for all $n > N_{0}$ sufficiently large, uniformly in $0 \leq k \leq n-2$.
\end{lemma}

\begin{proof}
We first prove the claim for $k= n-2$, the maximum value.Let $\mathcal{F} = \mathcal{G}_{n-2}$. 
We can write down the dependence of $\sup_{z} \rho(z)$ on $\mathcal{F}$ explicitly. Let $H = \mathrm{span}(v_{1},\ldots,v_{n-2})$ be the hyperplane spanned by $v_{1},\ldots,v_{n-2}$. Let $\mathcal{P}_{H} \, : \, \mathbb{R}^{n} \mapsto H$ be the operator associated with orthogonal projection onto $H$, and define
\be 
v^{0} &= \mathcal{P}_{H}(v_{n-1}) \\
v^{+} &= v_{n-1} -\mathcal{P}_{H}(v_{n-1}).
\ee 
We note that $\langle X, v^{0} \rangle$ and $v^{+}$ are both $\mathcal{F}$-measurable. Let $X_{H} = \| \mathcal{P}_{H} X \|_{2}^{2}$. The random variable $X_H$ is also $\mathcal{F}$-measurable.

Let $Z^{+} = \langle X, v^{+} \rangle$, so that 
\be \label{eqn:spheredistinlaw}
\mathcal{L}(Z) = \mathcal{L}(Z^{+} + \langle X, v^{0} \rangle|\mathcal{G}_k).
\ee
 Let $Y = \| v^{+} \|_{2} \, \sqrt{1 - X_{H}}\|S$, where $S \sim \mathrm{Unif}(\mathrm{S}^{(1)})$. Then $Z^{+} \stackrel{D}{=} Y[1]$. Thus, the density $\rho^{+}$ of $Z^{+}$ satisfies
\be 
\rho^{+}(z) &= \frac{2}{\pi} \frac{\sqrt{ \|v^{+} \|_{2}^{2} (1 - X_{H}) - z^{2}}}{ \|v^{+} \|_{2}^{2} (1 - X_{H})} \\
&\leq \frac{2}{\pi} \frac{1}{ \sqrt{ \|v^{+} \|_{2}^{2} (1 - X_{H})}}.
\ee 
By Equation \ref{eqn:spheredistinlaw} and the fact that $\langle X, v^{0} \rangle \in \mathcal{F}$, it follows that 
\be \label{EqDensityBoundCircleFullyConditional}
\sup_{z} \rho(z) = \sup_{z} \rho^{+}(z) \leq  \frac{2}{\pi} \frac{1}{ \sqrt{ \|v^{+} \|_{2}^{2} (1 - X_{H})}}.
\ee 
Thus, to complete the proof, it suffices to bound $\frac{1}{ \sqrt{ \|v^{+} \|_{2}^{2} (1 - X_{H})}}$ with high probability. 

First, we bound $1 - X_{H}$. Since $X_{H} = \| \mathcal{P}_{H} \|_{2}^{2}$ where $X \sim \mathrm{Unif}(\mathrm{S}^{(n-1)})$ and $H$ is an independently and randomly chosen hyperplane of dimension $n-2$, we can assume without loss of generality that $H = \{x \in \mathbb{R}^{n} \, : \, x[n-1] = x[n] = 0\}$. Thus,
\be \label{IneqCircleExactUnion1}
\P[ 1 - X_{H} \leq n^{-20}] &= \P[1 - X[1]^{2} - \ldots - X[n-2]^{2} \leq n^{-20}] \\
&= \P[X[n-1]^{2} + X[n]^{2} \leq n^{-20}] \\
&= O(n^{-15}).
\ee 
By exactly the same reasoning, we deduce that
\be \label{IneqCircleExactUnion2}
\P[ \| v^{+} \|_{2}^{2} \leq n^{-5}] &= \P[X[n]^{2} \leq n^{-20}] \\
&= O(n^{-15}).
\ee 
Combining Inequalities \eqref{IneqCircleExactUnion1} and \eqref{IneqCircleExactUnion2} with Inequality \eqref{EqDensityBoundCircleFullyConditional}, we conclude that 
\be 
\P[\sup_{z} \rho(z) > n^{-20}] = O(n^{-15}).
\ee 
This completes the proof of the lemma.
\end{proof}

Let $P_{1},\ldots,P_{N}$ be i.i.d. draws from the Haar measure on $\mathrm{SO}(N)$. Define a symmetric $N \times N$ matrix $D_{\infty}$ by 
\be [EqDefDUnif]
D_{\infty}[i,i] &= 1, \\
D_{\infty}[i,j] &= -\Tr[ a_i  (\prod_{\ell = i+1}^{j} P_{\ell})  a_{j}  (\prod_{\ell = i+1}^{j} P_{\ell})^{-1} ], , \qquad i < j \\
D_{\infty}[i,j] &= D_{\infty}[j,i], \qquad i > j.
\ee 

For $i < j$, let $P_{i,j} = \prod_{\ell = i+1}^{j} P_{\ell}$. We note that, for $1 \leq i < j \leq n$,

\be 
D_{\infty}[i,j] = -\Tr[ a_i  P_{i,j}  a_{j} (P_{i,j})^{-1} ]. \\
\ee 

\begin{remark} \label{RemWhyDInf}
The matrix $D_{\infty}$ has two useful properties. First, the appearance of the Haar measure in the definition of $D_{\infty}$ means that it is easier to make exact calculations involving the entries of $D_{\infty}$ than those of $D$. Second, $D_{\infty}$ is `close' to $D$, and so bounds on $D_{\infty}$ can easily be transferred to bounds on $D$. More precisely, Theorem 1 of \cite{Oliv07} implies that the entries of $D$ converge to those of $D_{\infty}$ as $\mathcal{Q}$ goes to infinity. The $\Theta(n^{2} \log(n))$ scaling of the lower bound $s_{i+1} - s_{i} \geq \mathcal{Q} n^{2} \log(n)$ in Definition \ref{DefSubsetChoice} was chosen so that we could guarantee that $D$ and $D_{\infty}$ must be close in distribution.
\end{remark}

For $1 \leq i < N$, we define $\mathcal{F}_{i}$ to be the $\sigma$-algebra generated by the matrices $\{P_{j}\}$ and the inner products  
$\{ -\Tr[ a_i  P_{i,j}  a_{j} (P_{i,j})^{-1} ] \}$
for $j > i + 1$.
\begin{lemma} \label{LemmaSatisfying}
The matrix $D_{\infty}$ and $\sigma$-algebras $\{\mathcal{F}_{i}\}_{1 \leq i <  n}$ given above satisfy Assumption \ref{assum:M} with 
\be 
C = N^{20}
\ee 
for all $N > N_{0}$ sufficiently large. Thus for all $N > N_0$ sufficiently large,
\be
\P[|D_\infty| < (4N^{21})^{-4(N+1)} ] \leq \frac{3}{N}.
\ee 
\end{lemma}
\begin{proof}

We make some initial observations. Let 
\be 
\mathfrak{S}_{n} = \{ a \in \lag \, : \, \| a \|_{\mathrm{HS}} = 1 \},
\ee 
and define the bijection $\mathcal{M}\, : \, \mathfrak{S}_{n} \mapsto \mathrm{S}^{(N-1)}$ by 
\be 
\mathcal{M}(a) = ( \langle a, a_{1} \rangle_{\mathrm{HS}}, \ldots, \langle a, a_{N} \rangle_{\mathrm{HS}}).
\ee 
We note that, if $P \sim \mathrm{Unif}(\so)$ and $a \in \mathfrak{S}_{n}$, then 
\be \label{EqSoSphere}
\mathcal{M}(P a P^{-1}) \sim \mathrm{Unif}(\mathrm{S}^{(N-1)}).
\ee 
We first prove Inequality \eqref{IneqBaseCaseDensity}. When $N$ is odd, this is trivial. When $N$ is even, we let $X \sim \mathrm{Unif}(\mathrm{S}^{(n-1)})$. By Equation \eqref{EqSoSphere}, we have that $D_{\infty}[N-1,N] \stackrel{D}{=} X[1]$, so
\be 
\P[| D_{\infty}[N,N]D_{\infty}[N-1,N-1] - D_{\infty}[N-1,N]^{2}| <  (4 N^{6})^{-4}] &= \P[|1 - X[1]^{2}| < (4 N^{6})^{-4}] \\
&= O(N^{-3}).
\ee 
We next prove Inequality \eqref{IneqRandomMatDensityAssumption}. Observe that:
\begin{enumerate}
\item $P_{i+1}$ is independent of $\{ P_{j} \}_{j > i+1}$.
\item By Equality \eqref{EqSoSphere}, the vectors $\{ \mathcal{M}( (P_{i+2} \ldots  P_{j}) a_{j} (P_{i+2} \ldots  P_{j})^{-1} )\}_{j > i+1}$ are i.i.d. choices from the sphere $S^{(N-1)}$, and the entries $\{D_{\infty}[i,j]\}_{j>i+1}$ are inner products of these vectors with $\mathcal{M}(a_{i})$:
\be 
D_{\infty}[i,j+1] &= -\Tr[ a_i  (\prod_{\ell = i+1}^{j} P_{\ell})  a_{j}  (\prod_{\ell = i+1}^{j} P_{\ell})^{-1} ] \\
&= -\langle \mathcal{M}(a_{i}), \mathcal{M}( (P_{i+2} \ldots  P_{j}) a_{j} (P_{i+2} \ldots  P_{j})^{-1} )^{\dag} \rangle.
\ee  
\item Combining the previous two observations, the distribution of 
\be 
Z \sim \mathcal{L}( -\langle \mathcal{M} ( a_{i} ), \mathcal{M}(P_{i+1} a_{i+1} P_{i+1}^{-1}) \rangle | \mathcal{F}_{i}) 
\ee 
satisfies the requirements of Lemma \ref{LemmaCondDensSphereBound}. 
\end{enumerate}
Thus, by Lemma  \ref{LemmaSmallBallBound}, the matrix $D_\infty$ satisfies Inequality \eqref{IneqRandomMatDensityAssumption} of Assumption \ref{assum:M} with constant $C = N^{20}$.
The conclusion follows immediately from Lemma \ref{LemmaGenMatrixDetBound}. This completes the proof.
\end{proof}

We apply our results to obtain a bound on the smallest singular value of $D_{\infty}$:

\begin{lemma} [Smallest Singular Values of $D_{\infty}$] \label{LemmaDetSing}
Let $D_{\infty}$ be the matrix defined as in \eqref{EqDefDUnif}, and let $\sigma_{1}(D_{\infty}) \leq \sigma_{2}(D_{\infty}) \leq \ldots \leq \sigma_{N}(D_{\infty})$ be its singular values. Then
\be 
\P[\sigma_{1}(D_{\infty}) \leq N^{-N} (4N^{21})^{-4(N+1)}] = o(1).
\ee 
\end{lemma}
\begin{proof}
Using the trivial bound $\sigma_{i}(D_{\infty}) \leq N \max_{k,\ell} |D_{\infty}[k,\ell]| \leq N$, we have
\be 
\sigma_{1}(D_{\infty}) &= |D_{\infty}| \, \prod_{i=2}^{n} \sigma_{i}(D_{\infty})^{-1} \\
&\geq | D_{\infty} | \, N^{-N}.
\ee 
Thus, for any $0 < r < \infty$.
\be 
\P[\sigma_{1}(D_{\infty}) \leq r] \leq \P[|D_{\infty}| \leq r N^{N}].
\ee
Choosing $r = N^{-N} (4N^{21})^{-4(N+1)}$ and applying Lemmas \ref{LemmaGenMatrixDetBound} and \ref{LemmaSatisfying}, we have 
\be 
\P[\sigma_{1}(D_{\infty}) \leq N^{-N} (4N^{21})^{-4(N+1)}] &\leq \P[ |D_{\infty}| \leq N^{-N} (4N^{21})^{-4(N+1)} N^{N}] \\
&= \P[|D_{\infty}| \leq (4N^{21})^{-4(N+1)}]  \\
&\leq \frac{3}{N}.
\ee 
This completes the proof.
\end{proof}

\subsection{Application to Kac's Walk} \label{SubsecApplicationsRandomMatKac}

We show that the sequence $\{ \phi_{n}^{-1} \}_{n \geq 1}$ defined in Equation \eqref{DefScalingSingVal} does not grow too quickly, and thus complete our proof of Theorem \ref{MainKacThm}. We do this by comparing the matrix $D$ of interest, defined in Equation \eqref{EqJacMatComp}, with the matrix $D_{\infty}$ studied in Section \ref{SecToyApplicationRandMat}.

\begin{lemma} \label{TangentMapIsNotSmall}
Let $\mathcal{Q} = 100$. Then the sequence $\{ \phi_{n}^{-1} \}_{n \geq 1}$ defined in Equation \eqref{DefScalingSingVal} and associated with the coupling defined in Section \ref{DefCoupLongDesc} satisfies 
\be 
\phi_{n}^{-1} \leq (4N^{24})^{5N} 
\ee 
for all $n > N_{0}$ sufficiently large.
\end{lemma} 

\begin{remark}
We believe that Lemma \ref{TangentMapIsNotSmall} may hold with $\phi_{n} = O(n^{k})$ for some $k < \infty$. 
\end{remark}

\begin{proof}

We begin by relating the symmetric matrices $D$, $D_{\infty}$. Recall the law of the off-diagonal entries of $D$ given in Equation \eqref{EqJacMatComp}:
\be \label{EqJacMatCompRedux}
D[i,j] = -\Tr[ a_i  (\prod_{\ell = i+1}^{j} M_{\ell})  a_{j}  (\prod_{\ell = i+1}^{j} M_{\ell})^{-1} ], \qquad i<j,
\ee 
where
\be 
M_{\ell} = \prod_{t=s_{\ell}+1}^{s_{\ell+1}} R(i_{s},\theta_{s}).
\ee 
We will relate $D$ to $D_{\infty}$. Recall that the off-diagonal entries of $D_{\infty}$ are written
\be 
D_{\infty}[i,j] = -\Tr[ a_i  (\prod_{\ell = i+1}^{j} M_{\ell}')  a_{j}  (\prod_{\ell = i+1}^{j} M_{\ell}')^{-1} ], \qquad i<j,
\ee 
where $\{M_{\ell}\}$ are an i.i.d. sequence from the Haar measure on $\so$. For fixed $1 \leq i < N$, let $\mathcal{G}_{i} = \mathcal{G}_{i}(\{m_{k,\ell}\}) $ ( respectively $\mathcal{G}_{i}'(\{m_{k,\ell}\})$) be the event that $D[k,\ell] = m_{k,\ell}$ (respectively $D_{\infty}[k,\ell] = m_{k,\ell}$) for all $k,\ell \geq i+2$. By the main result of \cite{Oliv07} and Lemma \ref{LemmaDifProdSwap}, we have 
\be \label{LemmaKacEmbedding}
W_{2}( \mathcal{L}(D[i,(i+1):n] \, | \, \mathcal{G}_{i}), \mathcal{L}(D_{\infty}[i,(i+1):n] \, | \, \mathcal{G}_{i}') )^{2} \leq 2 n^{-\mathcal{Q} + 3}
\ee 
for all $\mathcal{Q} \geq 5$ and all $n > N_{0}$ sufficiently large.

With this initial calculation complete, we can now prove Lemma \ref{TangentMapIsNotSmall}. Fix notation as in Lemma \ref{LemmaKacEmbedding}. We will apply Lemma \ref{LemmaRandomMatComparison}, with $D_{\infty}$ playing the role of $M$ and $D$ playing the role of $M'$. By Lemma \ref{LemmaSatisfying}, $D_{\infty}$ satisfies the conditions of Lemma \ref{LemmaRandomMatComparison} with constant $C = N^{20}$. By Inequality \ref{LemmaKacEmbedding}, $D_{\infty}, D$ satisfy Condition \eqref{RandomMatCompWassClosenessCond} of Lemma \ref{LemmaRandomMatComparison} for all $\mathcal{Q} > 94$. Thus, for fixed $\mathcal{Q} > 94$,
\be 
\P[|D| > (4N^{23})^{5N}] \leq \frac{3}{n}
\ee 
for all $n > N_{0}$ sufficiently large. By a calculation identical to that in Lemma \ref{LemmaDetSing}, we conclude that 
\be 
\P[\sigma_{1}(D) > (4N^{24})^{5N}] \leq \frac{3}{n},
\ee 
completing the proof.
\end{proof}

\section{Proof of Theorem \ref{MainKacThm}}

Inequality \eqref{IneqMainThmLower}, the lower bound in Theorem \ref{MainKacThm}, is exactly the main conclusion of Theorem \ref{LemmaLowerBound}.

To prove Inequality \eqref{IneqMainThmUpper}, the upper bound in Theorem \ref{MainKacThm}, we recall from Lemma \ref{TangentMapIsNotSmall} that the sequence $\{ \phi_{n}^{-1} \}_{n \geq 1}$ defined in Equation \eqref{DefScalingSingVal}  satisfies 
\be 
\phi_{n}^{-1} \leq (4N^{24})^{5N} = o(n^{n^{2.2}})
\ee 
whenever $\mathcal{Q} > 100$.

Thus, in Theorem \ref{MainKacThmStrengthening}, we may take $\mathcal{Q} = 101$, $R = 808 n^{4} \log(n)$ and $\log(\phi_{n}) \leq 120 \, n^{2} \, (\log(4) + \log(n))$. Thus, by Theorem \ref{MainKacThmStrengthening}, we have 
\be 
\limsup_{n \rightarrow \infty} \sup_{X_{0} = x \in \so} \| \mathcal{L}(X_{t}) - \mu \|_{\TV} = 0
\ee 
for any sequence $t = t(n)$ satisfying 
\be 
t(n) \geq 10^{7} \, n^{4} \log(n) \geq 808 n^{4} \log(n) + 5 n^{2} \log(n) + 900 n^{2} \times (120 \, n^{2} \, (\log(4) + \log(n))).
\ee 
This completes the proof of Inequality \eqref{IneqMainThmUpper}.

\section{Discussion}

Our work leaves open the question as to whether the mixing time of Kac's walk is indeed $\Theta(n^{2} \log(n))$ as conjectured, and whether it exhibits the \textit{cutoff phenomenon}. We cannot obtain such a bound by more careful analysis of the terms in Theorem \ref{MainKacThmStrengthening}. However, we believe that it is possible to obtain the desired bound by a more careful analysis of the terms in Theorem \ref{MainKacThmOptimistic}.

The main difficulty in applying our method is obtaining a bound on $\phi_{n}$, which measures the smallest singular value $\sigma_{1}(D)$ of the matrix $D$. In this paper, we were only able to bound $\sigma_{1}(D)$ by comparing $D$ to a simpler limiting matrix $D_{\infty}$ for which exact calculations were available. To improve this bound further, we believe that it is necessary to analyze $\sigma_{1}(D)$ directly. To obtain  an $O(n^{2} \log(n))$ bound on the mixing time of Kac's walk using our method, it would be enough to obtain any polynomial bound $\phi_{n} = O(n^{k})$ for some $0 < k < \infty$. The main obstacles to proving such a bound are:

\begin{enumerate}
\item Our weak random matrix bound in Lemma \ref{LemmaGenMatrixDetBound}. Our argument for Lemma \ref{LemmaGenMatrixDetBound}, like that in \cite{friedland2013simple}, only takes advantage of the randomness of entries within distance 1 of the diagonal of an $n$ by $n$ random matrix $M$. Unfortunately, no argument that only analyzes these entries can give any bound that is stronger than $|M|^{-2} = 2^{O(n)}$. Since the matrix $D$ of interest has `many more' than $3n$ `pieces' of randomness, we can hope to take advantage of them and obtain a stronger bound, as in \cite{farrell2015smoothed}.
\item We bound the determinant $|D|$ of the random matrix $D$ and use this to obtain a bound on the smallest singular value $\sigma_{1}(D)$; we give up a factor of $N^{N-1}$ in the process (see Lemma \ref{LemmaDetSing}). To avoid this loss, we must either obtain stronger bounds on the joint distribution of the remaining singular values $\sigma_{2}(D) \leq \sigma_{3}(D) \leq \ldots \leq \sigma_{N}(D) \gtrsim \log(n)$, or to bound $\sigma_{1}(D)$ directly as in \cite{farrell2015smoothed}.
\item Our bound on $\sigma_{1}(D)$ is obtained through a comparison of $D$ and the related matrix $D_{\infty}$ (see Theorem \ref{TangentMapIsNotSmall}). While $D$, $D_{\infty}$ are nearby under the coupling studied in Theorem \ref{MainKacThmStrengthening}, they are very far under the coupling studied in  Theorem \ref{MainKacThmOptimistic}. Thus, to get a better mixing bound, the matrix $D$ should be studied directly.
\end{enumerate}

It seems possible to extend our arguments to resolve any two of these three obstacles together; however, we see no route to resolving all three simultaneously.

Our approach can be applied to the analysis of other Gibbs samplers on constrained or non-convex state spaces. As mentioned in Remark \ref{RemarkCoupGenerality}, our non-Markovian coupling can be defined for generic Gibbs samplers. The key ingredients in an analysis of the entire coupling are the analysis of an underlying `scaffolding' coupling $\{ \hat{X}_{t}, \hat{Y}_{t}\}_{t \geq 0}$, the analysis of the smallest singular value of the Jacobian of the `perturbation map' $f_{\mathcal{A},\epsilon_{n}}$, and soft bounds on the smoothness of $f_{\mathcal{A},\epsilon_{n}}$.

\section*{Acknowledgements}
NSP is partially supported by an ONR grant and AMS is partially supported by an NSERC grant.  We thank Tristan Collins, Persi Diaconis, John Jiang, Federico Poloni, Terence Tao and Roman Vershynin for helpful conversations and suggestions.
\bibliographystyle{alpha}
\bibliography{KacBib}

\newpage 

\appendix

\section{Proofs of Technical Bounds} \label{SecAppProofsBoring}

We prove the bounds in Section \ref{SecBigMainLemmas}.

\subsection{Matrix Estimates}

\begin{proof} [Proof of Lemma \ref{LemmaDifProdSwap}]
We note the `telescoping sum' identity
\be 
\prod_{i=1}^{k} Q_{i} - \prod_{i=1}^{k} P_{i} = \sum_{i=1}^{k} (\prod_{\ell=1}^{i-1} Q_{\ell} ) (Q_{i} - P_{i})(  \prod_{\ell=i+1}^{k} P_{\ell}).
\ee 
The result then follows immediately from application of the triangle inequality and the inequality $\| ABC \|_{\hs} \leq \| A \|_{\mathrm{Op}}  \| B \|_{\hs} \| C \|_{\mathrm{Op}} $ for any $A,B,C \in M(n)$. 
\end{proof}

\begin{proof} [Proof of Lemma \ref{LemmaJacobianBound}]
By Fact 4 in chapter 15 of \cite{hogben14},
\be 
\sum_{i=1}^{n} (\sigma_{i}(M_{1}) - \sigma_{i}(M_{2}))^{2} \leq \sum_{i=1}^{n} \sigma_{i}(M_{1} - M_{2})^{2}.
\ee 
By assumption \eqref{LemmaJacBoundMainReq}, this implies 
\be 
\sum_{i=1}^{n} (\sigma_{i}(M_{1}) - \sigma_{i}(M_{2}))^{2} \leq  N \delta^{2} \sigma_{1}(M_{1})^{2}.
\ee 
In particular,
\be 
\max_{1 \leq i \leq N} | \sigma_{i}(M_{1}) - \sigma_{i}(M_{2}) | \leq \sqrt{N} \delta \sigma_{1}(M_{1}).
\ee 
For a symmetric matrix $M$, let $\lambda_{1}(M), \ldots, \lambda_{N}(M)$ denote the eigenvalues, ordered so that $| \lambda_{i}(M) = \sigma_{i}(M)$. We have
\be 
|\frac{\mathrm{det}(M_{2})}{\mathrm{det}(M_{1})} - 1|  &= | \frac{\prod_{i=1}^{N} \lambda_{i}(M_{2})}{\prod_{i=1}^{N} \lambda_{i}(M_{1})} - 1 | \\
&= | \prod_{i=1}^{N} \frac{ \lambda_{i}(M_{1}) + (\lambda_{i}(M_{2}) - \lambda_{i}(M_{1}))}{\lambda_{i}(M_{1})}  - 1| \\
&= \prod_{i=1}^{N} | \frac{ (\lambda_{i}(M_{2}) -\lambda_{i}(M_{1}))}{\lambda_{i}(M_{1})} | \\
&\leq \prod_{i=1}^{N} | \frac{ \sqrt{N} \delta \sigma_{1}(M_{1}) }{\sigma_{1}(M_{1})} \\
&= N^{\frac{N}{2}} \delta^{N}
\ee 
and the proof is finished.
\end{proof}
\begin{proof} [Proof of Lemma \ref{LemmaTangentClose}] 
Let $x = (x_{1},\ldots,x_{N}), y = (y_{1},\ldots,y_{N})$. By Equation \eqref{EqGenFuncDer} it follows that
\be
\| \d f_x(h) &- \d f_y(h)\|_\hs \\
&= \| \sum_{j=1}^N \prod_{k =1}^N  \R_k \,e^{(\t_k + x_k) a_k} (\I + (h_ka_k - \I)1_{k=j})
 -  \sum_{j=1}^N \prod_{k =1}^N  \tilde{\R}_k \,e^{(\tt_k + y_k) a_k} (\I + (h_ka_k - \I)1_{k=j}) \|_{\hs} \\
&\leq \sum_{j=1}^{N} \| \prod_{k =1}^N  \R_k \,e^{(\t_k + x_k) a_k} (\I + (h_ka_k - \I)1_{k=j}) - \prod_{k =1}^N  \tilde{\R}_k \,e^{(\tt_k + y_k) a_k} (\I + (h_ka_k - \I)1_{k=j}) \|_{\hs} \\
&\leq \sum_{j,k=1}^{N} \| \R_k \,e^{(\t_k + x_k) a_k} (\I + (h_k a_k - \I)1_{k=j})  - \tilde{\R}_k \,e^{(\tt_k + y_k) a_k} (\I + (h_k a_k - \I)1_{k=j}) \|_{\hs} \\
&\leq 2 N^{2} \max( \| R_{k} - \tilde{R}_{k} \|_{\hs}, \| e^{(\t_k + x_k) a_k} - e^{(\tt_k + y_k) a_k} \|_{\hs}) \\
&\leq 4 N^{2} \max_{1 \leq j \leq N} \max( | x_{k} |, |y_{k}|, |\t_k - \tt_k|) \leq  4 N^{2} c,
\ee 
where the third and fourth lines are both applications of Lemma \ref{LemmaDifProdSwap}. Applying Lemma \ref{LemmaDifProdSwap} once more, we have
\be 
\| \d\L_{f(y) (f(x))^{-1}} \d f_x(h) - df_y(h) \|_\hs &\leq \| f(y) - f(x) \|_{\hs} + \| \d f_x(h) - \d f_y(h)\|_\hs,  \\
&\leq 8 N^{2} c
\ee 
the second part of inequality \eqref{IneqTanCloseTwoConc}. This completes the proof.
\end{proof}

\begin{proof} [Proof of Lemma \ref{ClosenessOfExp}]
Let $x = (x_{1},\ldots,x_{N}) \in [-c,c]^{N}$. We calculate:
\be 
\| &f(x) - f(0) \,\exp(\d\L_{f(0)^{-1}} \d f_{0}(x)) \|_{\hs} \\
&=\| \prod_{k=1}^{N} \R_{k} e^{(\t_{k} + x_{k})a_{k}} \\
&\hspace{2cm} - \prod_{k=1}^{N} \R_{k} e^{\t_{k} a_{k}}  \exp( (\prod_{k=1}^{N} \R_{k} e^{\t_{k} a_{k}} )^{-1} \sum_{i=1}^{N} \prod_{k =1}^N  \R_k \,e^{(\t_k + x_k) a_k} (\I + (h_k a_k - \I)1_{k=i})) \|_{\hs} \\
&= \| \prod_{k=1}^{N} \R_{k} e^{(\t_{k} + x_{k})a_{k}} \\
&\hspace{2cm} - \prod_{k=1}^{N} \R_{k} e^{\t_{k} a_{k}} \sum_{u=0}^{\infty} \frac{1}{u!} ( (\prod_{k=1}^{N} \R_{k} e^{\t_{k} a_{k}} )^{-1} \sum_{i=1}^{N} \prod_{k =1}^N  \R_k \,e^{(\t_k + x_k) a_k} (\I + (h_k a_k - \I)1_{k=i}))^{u} \|_{\hs} \\
&=  \| \prod_{k=1}^{N} \R_{k} e^{\t_{k} a_{k}} \sum_{u=0}^{\infty} \frac{(x_{k} a_{k})^{u}}{u!} \\
&\hspace{2cm}  -  \prod_{k=1}^{N} \R_{k} e^{\t_{k} a_{k}} \sum_{u=0}^{\infty} \frac{1}{u!} ( (\prod_{k=1}^{N} \R_{k} e^{\t_{k} a_{k}} )^{-1} \sum_{i=1}^{N} \prod_{k =1}^N  \R_k \,e^{(\t_k + x_k) a_k} (\I + (h_k a_k - \I)1_{k=i}))^{u} \|_{\hs} \\
&\leq 8 N^{2} \max_{1 \leq k \leq N} | x_{k} |^{2}  \\
&\hspace{1cm} + \| \sum_{i=1}^{N} \prod_{k=1}^{N} \R_{k} e^{\t_{k} a_{k}} (\I + (h_k a_k - \I)1_{k=i}) -  \sum_{i=1}^{N} \prod_{k =1}^N  \R_k \,e^{\t_k a_k} (\I + (h_k a_k - \I)1_{k=i}) \|_{\hs} \\
&= 8 N^{2} \max_{1 \leq k \leq N} | x_{k} |^{2} \leq 8 N^{2} c^{2},
\ee 
where the second-last line relies on the triangle inequality and repeated application of Lemma \ref{LemmaDifProdSwap} to remove all terms that are of second or higher order in $\{ a_{k} \}_{1 \leq k \leq N}$. This completes the proof of the lemma. 
\end{proof}

\begin{proof} [Proof of Lemma \ref{LemmaTangentSize}]

By Equation \eqref{EqGenFuncDer},
\be \label{EqDiffSizeCalc1}
\langle & \d f_{x}(h),\d f_{x}(h') \rangle_{\hs} = \Tr [\d f_{x}(h) \d f_{x}(h')^{\dag}] \\
&= \Tr\Big[\Big(\sum_{i=1}^N \prod_{k =1}^N  \R_k \,e^{(\t_k + x_k) a_k} (\I +(h_k a_k - \I)1_{k=i})\Big) \Big(\sum_{i=1}^N \prod_{k =1}^N  \R_k \,e^{(\t_k + x_k) a_k} (\I + (h_{k}' a_k - \I)1_{k=i})\Big)^{\dag}\Big] \\
&= \sum_{i=1}^{N} \Tr\Big[\prod_{k =1}^N  \R_k \,e^{(\t_k + x_k) a_k} (\I + (h_ka_k - \I)1_{k=i}) \prod_{k=N}^{1} (\I - (h_{k}' a_k - \I)1_{k=i})  e^{-(\t_k + x_k) a_k} \R_k^{-1} \Big] \\ 
&\hspace{0.3in}+ \sum_{1 \leq i \neq j \leq N} \Tr \Big[\prod_{k =1}^N  \R_k \,e^{(\t_k + x_k) a_k} (\I + (h_ka_k - \I)1_{k=i}) \prod_{k=N}^{1} (\I - (h_{k}' a_k - \I)1_{k=j}) e^{-(\t_k + x_k) a_k} \R_{k}^{-1} \Big] \\
&\equiv \sum_{i=1}^{N} S_{i} + \sum_{1 \leq i \neq j \leq N} S_{ij}.
\ee 
We calculate terms of the form $S_{j}$ and $S_{ij}$ separately. For any $S_{i}$, applying the cyclic permutation property of the trace operator yields
\be 
S_{i} &= \Tr\Big[\prod_{k =1}^N  \R_k \,e^{(\t_k + x_k) a_k} (\I + (h_ka_k - \I)1_{k=i}) \prod_{k=N}^{1} (\I - (h_{k}' a_k - \I)1_{k=i}) e^{-(\t_k + x_k) a_k} \R_{k}^{-1}\Big]  \\
&= \Tr \Big[\R_{i} (h_{i} a_{i})(-h_{i}' a_{i}) \R_{i}^{-1}] = - h_{i}h_{i}' \Tr[ \R_{i} a_{i}^{2} \R_{i}^{-1}\Big] = -h_{i}h_{i}' \Tr[a_{i}^{2}] =  h_{i} h_{i}'.\label{EqDiffSizeCalc2}
\ee  
For any $S_{ij}$ with $i < j$,
\be 
S_{ij} &= \Tr \Big[\prod_{k =1}^N  \R_k \,e^{(\t_k + x_k) a_k} (\I + (h_{k} a_k - \I)1_{k=i}) \prod_{k=N}^{1} (\I - (h_{k}' a_k - \I)1_{k=j}) e^{-(\t_k + x_k) a_k} \R_{k}^{-1} \Big]  \\
&= -h_{i}h_{j}' \Tr[ a_i  M_{ij} \R_{j} \,e^{(\t_j + x_j) a_j}  a_{j}  \,e^{-(\t_j + x_j) a_j} \R_{j}^{-1} M_{i,j}^{-1} ] \\
&=  -h_{i}h_{j}' \Tr[ a_i  M_{i,j} \R_{j}  a_{j}  \R_{j}^{-1} M_{i,j}^{-1} ] \\
&= h_{i} h_{j}' D_{ij}. \label{EqDiffSizeCalc3}
\ee

For $j < i$, a similar calculation gives 
\be 
S_{ij} = h_{i} h_{j}' D_{i,j}.
\ee 
Combining Equalities \eqref{EqDiffSizeCalc1}, \eqref{EqDiffSizeCalc2} and \eqref{EqDiffSizeCalc3} completes the proof.  
\end{proof}

\begin{proof} [Proof of Lemma \ref{LemmaDiff}]
We note that $f$ is clearly smooth, and so we must only check that it is bijective. This result will follow almost immediately from Lemma 10 of \cite{raymond2002local} and our bounds in Lemmas \ref{ClosenessOfExp} and \ref{LemmaTangentSize}.

We begin to set up notation. For a point $x$ in a metric space $(\Omega,d)$ and constant $\delta > 0$, let $\mathcal{B}_{\delta}(x) = \{ y \in \Omega \, : \, d(x,y) \leq \delta \}$ be the ball of radius $\delta$ around $x$. We then define a map $\mathcal{F}$ from $\mathcal{B}_{n^{-6} \phi_{n}}(\I) \subset \so$ to $\mathbb{R}^{N}$ as follows. Let $x \in \mathcal{B}_{n^{-6} \phi_{n}}(\I)$. Since the $\exp$ map is surjective and sends lines to geodesic curves, we can write $x = \exp(h)$ for some $h \in \mathfrak{so}(n)$ with $\| h \|_{\hs} \leq 2 n^{-6} \phi_{n}$. Furthermore, we can write $h = \sqrt{2} \sum_{i=1}^{N} h_{i} a_{i}$. We then define $\mathcal{F}(x) = (h_{1},h_{2},\ldots,h_{N})$. Finally, we define the map $g = \mathcal{F} \circ f: [-c,c]^{N} \mapsto \mathbb{R}^{N}$. 

We point out that $\mathcal{F}$ has small distortion: for $x,y \in \so$ with $\mathcal{F}(x) = h_{x}$, $\mathcal{F}(y) = h_{y}$,
\be \label{IneqDistIneq}
\| x - y \|_{\hs} = \| \exp(h_{x}) - \exp(h_{y}) \|_{\hs} = \| h_{x} - h_{y} \| + O(N^{2} \|h_{x} - h_{y} \|^{2}).
\ee 

We now obtain the estimates required to use Lemma 10 of \cite{raymond2002local}. Following the notation of that paper, we set $\rho = \frac{1}{256} n^{-6} \phi_{n}$, $\delta = \frac{ \phi_{n}}{8}$ and $\rho_{\ast} = \frac{\delta}{8} \rho$. For $x,z \in [-c,c]^{N}$ with $\| x - z \| = \rho$, 

\be \label{IneqDiffeoLongCalc}
\| g(x) - g(z) \|_{\hs} 
&= \| g(x) - \mathcal{F}(f(0)\, \exp(\d\L_{f(0)^{-1}} \d f_{0}(x))) + \mathcal{F}(f(0) \exp(\d \L_{f(0)^{-1}} \d f_{0}(x))) \\
&\hspace{1cm}  + \mathcal{F}(f(0) \exp(\d\L_{f(0)^{-1}} \d f_{0}(z))) - \mathcal{F}(f(0) \exp(\d\L_{f(0)^{-1}} \d f_{0}(z))) - g(z) \|_{\hs} \\ 
&\geq \| \mathcal{F}( f(0) \exp(\d\L_{f(0)^{-1}} \d f_{0}(x))) - \mathcal{F}( f(0) \exp(\d \L_{f(0)^{-1}} \d f_{0}(z))) \|_{\hs}\\
&\hspace{1cm} - \| \mathcal{F}(f(x)) - \mathcal{F}(f(0) \exp(\d\L_{f(0)^{-1}} \d f_{0}(x))) \|_{\hs} \\
&\hspace{2cm} - \| \mathcal{F}(f(z)) - \mathcal{F}(f(0) \exp(\d\L_{f(0)^{-1}} \d f_{0}(z))) \|_{\hs} \\
&\geq \| \mathcal{F}( f(0) \exp(\d\L_{f(0)^{-1}} \d f_{0}(x))) - \mathcal{F}( f(0) \exp(\d \L_{f(0)^{-1}} \d f_{0}(z))) \|_{\hs} \\
&\hspace{1cm}- 2 \| f(x) - f(0) \exp(\d\L_{f(0)^{-1}} \d f_{0}(x)) \|_{\hs} \\
&\hspace{2cm} - 2 \| f(z) - f(0) \exp(\d\L_{f(0)^{-1}} \d f_{0}(z)) \|_{\hs} \\
&\geq \| f(0) \exp(\d\L_{f(0)^{-1}} \d f_{0}(x)) - f(0)\, \exp(\d L_{f(0)^{-1}} \d f_{0}(z)) \|_{\hs} - 32 N^{2} c^{2},
\ee 
where the second-last inequality follows from inequality \eqref{IneqDistIneq} and  Lemma \ref{ClosenessOfExp}, and the last inequality is due to Lemma \ref{ClosenessOfExp}. By Inequalities \eqref{IneqTangentMapNotSmall} and  \eqref{IneqDistIneq},
\be \label{IneqDiffeoShortCalc99}
 \| \mathcal{F}( f(0) \exp(\d\L_{f(0)^{-1}} \d f_{0}(x))) &- \mathcal{F}( f(0) \exp(\d \L_{f(0)^{-1}} \d f_{0}(z))) \|_{\hs} \\
 &\geq  \frac{\phi_{n}}{4} \| x - z \| - O(N^{3} \| x -z \|^{2}).
\ee 
Combining Inequalities \eqref{IneqDiffeoLongCalc} and \eqref{IneqDiffeoShortCalc99}, we conclude 
\be 
\| g(x) - g(z) \|_{\hs} \geq \frac{ \phi_{n}}{8} \|x -z \|.
\ee 
This proves the first condition of Lemma 10 of \cite{raymond2002local}: $\{ \| x-z\|_{\hs} > \rho \}$ implies that  $\{ \| g(x) - g(z) \| > \delta \rho \}$.  The second condition of Lemma 10 of \cite{raymond2002local} follows immediately from the second part of inequality \eqref{IneqTangentMapNotSmall}. Thus, $g$ satisfies the requirements of Lemma 10 of \cite{raymond2002local} with $\rho$, $\rho_{\ast}$, $\delta$ as above, and so $g$ is an injective map. But this implies that $f$ is injective as well, completing the proof.
\end{proof}

\subsection{Probability Estimates}

\begin{proof} [Proof of Lemma \ref{LemmaDimensionTimeGreedy}]
We recognize that $s_{N}$ is exactly the time required to collect all $N$ coupons in the standard `coupon collector problem' with $N$ coupons. Our Equation \eqref{EqCoupConLim} above is Equation (2) of \cite{ErRo61Coupon}.
\end{proof}

\begin{proof} [Proof of Lemma \ref{LemmaDimensionTime}]
We recognize that $s_{N} - \mathcal{Q} N^{2} \lceil \log(N) \rceil$ has exactly negative binomial distribution with both parameters equal to $N$. The bound is then a standard tail bound for the negative binomial distribution (see, \textit{e.g.}, the calculation in  \cite{ConcIneqWaste}).
\end{proof}

\begin{proof} [Proof of Lemma \ref{LemmaPathCloseness}]
For fixed $0 \leq t \leq n^{5}$, let $\mathcal{E}(t)$ be the event $\{ \sup_{0 \leq s \leq t} \| \hat{X}_{s} - \hat{Y}_{s} \|_{\hs} \leq \| \hat{X}_{0} - \hat{Y}_{0} \|_{\hs} n^{5 + C }\}$. By the main calculation on pp. 1216 of \cite{Oliv07}, we have for all $0 \leq t \leq n^{5}$ that
\be 
\E[\| \hat{X}_{t+1} - \hat{Y}_{t+1} \|_{\hs}  \textbf{1}_{\mathcal{E}(t)} ] \leq   \| X_{0} - Y_{0} \|_{\hs} (1 + \sqrt{\| \hat{X}_{0} - \hat{Y}_{0} \|_{\hs}} n^{7 + C })^{t} .  
\ee 
Inequality \eqref{IneqPathClosenessConc1} follows immediately from Markov's inequality and a union bound over $1 \leq t \leq n^{5}$. 

Inequality \eqref{IneqPathClosenessConc2} follows from noting that, 
\be 
\| X_{t} - \hat{X}_{t} \|_{\hs} &= \| \prod_{s=0}^{t-1} \R(i_{s}(x), \eta_{s}(x)) - \prod_{s=0}^{t-1} \R(i_{s}(x), \theta_{s}(x)) \|_{\hs} \\
&\leq \sum_{u=0}^{t-1} \| (\R(i_{u}(x),\eta_{t}(x)) - \R(i_{u}(x), \theta_{u}(x))) \prod_{s=u+1}^{t-1} \R(i_{s}(x), \eta_{s}(x)) \|_{\hs} \\
&= \sum_{u=0}^{t-1} \| (\R(i_{u}(x),\eta_{u}(x)) - \R(i_{u}(x), \theta_{u}(x))) \|_{\hs} \\
&\leq \sum_{u=0}^{t-1} 6 \| \eta_{u}(x) - \theta_{u}(x) \| \\
&\leq  6 t  \epsilon_{n} \leq 6 n^{5} \epsilon_{n}, 
\ee 
where we use Lemma \ref{LemmaDifProdSwap} in the second line. Inequality \eqref{IneqPathClosenessConc3} can be proved analogously. 
\end{proof}
\end{document}